\documentclass[a4paper,11pt]{article}

\usepackage[dvipsnames]{xcolor}
\usepackage[utf8]{inputenc} 
\usepackage{comment}
\usepackage{enumerate}   
\usepackage{enumitem}
\usepackage{amsmath}
\usepackage{amssymb}
\usepackage{amsthm}
\usepackage{thmtools}
\usepackage{url}
\usepackage{caption}
\usepackage{subcaption}
\usepackage[top=21 mm, bottom=22 mm, left=22 mm, right= 22 mm]{geometry}
\usepackage[bookmarks=false,hidelinks]{hyperref}

\usepackage{tikz}
\usetikzlibrary{math,calc,intersections}
\usetikzlibrary{positioning,arrows,shapes,decorations.markings,decorations.pathreplacing,matrix,patterns}
\tikzstyle{vertex}=[circle,draw=black,fill=black,inner sep=0,minimum size=5pt,text=white,font=\footnotesize]

\declaretheorem[name=Theorem,numberwithin=section]{theorem}

\newtheorem{lemma}[theorem]{\bf Lemma}

\newtheorem{proposition}[theorem]{\bf Proposition}

\newtheorem*{theorem*}{\bf Theorem}

\theoremstyle{definition}

\newtheorem{definition}[theorem]{\bf Definition}

\def\eps{\varepsilon}

\def\cF{\mathcal{F}}

\def\cH{\mathcal{H}}

\def\cN{\mathcal{N}}
\def\cP{\mathcal{P}}
\def\cQ{\mathcal{Q}}
\def\cS{\mathcal{S}}
\def\cV{\mathcal{V}}
\def\bF{\mathbb{F}}

\def\oF{\overline{\mathbb{F}}}
\def\bE{\mathbb{E}}
\def\bR{\mathbb{R}}
\def\bP{\mathbb{P}}
\def\Pb{\mathbb{P}}

\DeclareMathOperator{\aff}{Aff}
\DeclareMathOperator{\Var}{Var}
\DeclareMathOperator{\Cov}{Cov}
\DeclareMathOperator{\diag}{diag}

\DeclareMathOperator{\rs}{rs}

\title{\vspace{-0.9cm} Incidence bounds via extremal graph theory}
\author{Aleksa Milojevi\'c\thanks{ETH Zurich, e-mail: \textbf{\{aleksa.milojevic, benjamin.sudakov\}@math.ethz.ch}. Research supported in part by SNSF grant 200021\_196965.}, Benny Sudakov\footnotemark[1], Istv\'an Tomon\thanks{Ume\r{a} University, \emph{e-mail}: \textbf{istvan.tomon@umu.se}.}}
\date{}

\begin{document}

\maketitle
\begin{abstract}
The study of counting point-hyperplane incidences in the $d$-dimensional space was initiated in the 1990's by Chazelle and became one of the central problems in discrete geometry. It has interesting connections to many other topics, such as additive combinatorics and theoretical computer science. Assuming a standard non-degeneracy condition, i.e., that no $s$ points are contained in the intersection of $s$ hyperplanes, the currently best known upper bound on the number of incidences of $m$ points and $n$ hyperplanes in $\mathbb{R}^d$ is $$O_{d, s}((mn)^{1-1/(d+1)}+m+n).$$ This bound by Apfelbaum and Sharir is based on geometrical space partitioning techniques, which apply only over the real numbers.

In this paper, we propose a novel combinatorial approach to study such incidence problems over arbitrary fields. Perhaps surprisingly, this approach matches the best known bounds for point-hyperplane incidences in $\mathbb{R}^d$ for many interesting values of $m, n, d$, e.g. when $m=n$ and $d$ is odd. Moreover, in finite fields our bounds are sharp as a function of $m$ and $n$ in every dimension. We also study the size of the largest complete bipartite graph in point-hyperplane incidence graphs with a given number of edges and obtain optimal bounds as well.

Additionally, we study point-variety incidences and unit-distance problem in finite fields, and give tight bounds for both problems under a similar non-degeneracy assumption. We also resolve Zarankiewicz type problems for algebraic graphs. Our proofs use tools such as induced Tur\'an problems, VC-dimension theory, evasive sets and Hilbert polynomials. Also, we extend the celebrated result of R\'onyai, Babai and Ganapathy on the number of zero-patterns of polynomials to the context of varieties, which might be of independent interest.
\end{abstract}

\section{Introduction}

The Szemer\'edi-Trotter theorem \cite{SzT} is a fundamental result in combinatorial geometry, giving a sharp upper bound on the number of incidences in point-line configurations. It states that $m$ points and $n$ lines on the real plane determine at most $O((mn)^{2/3}+m+n)$ incidences, and as Erd\H{o}s showed (see e.g. \cite{E87}), this bound is the best possible. This deep result found numerous applications and inspired a large number of generalizations and extensions over the last decades. For example, the problem of counting point-line incidences in $\bR^3$ under certain non-degeneracy conditions has been studied by Guth and Katz \cite{GK} in their solution to the Erd\H{o}s distinct distances problem. The Szemer\'edi-Trotter theorem has also been applied by Elekes \cite{Elekes} to derive sum-product estimates over the reals. For these and other applications of incidence theorems, we refer the reader to the survey \cite{Dvir12}.

Extending incidence results to finite fields is more tricky, since standard space partitioning techniques do not apply anymore. The importance of point-line incidence bounds in $\bF_p^2$ stems from their close connection to sum-product estimates, as shown in the pioneering work of Bourgain, Katz and Tao \cite{BKT}. Following this work, point-line incidence bounds over finite fields have been extensively studied, see e.g. \cite{Kollar, MS21, SdZ17}.

In this paper, we consider several natural incidence questions in higher dimensions. The first question we address is to find the maximum number of incidences between $m$ points and $n$ hyperplanes in $\mathbb{F}^d$, for an arbitrary field $\bF$. In dimension 3 and above, it is possible that all $n$ hyperplanes contain all $m$ points: take $m$ points on a line, and $n$ hyperplanes containing this line. In order to avoid such trivialities, it is standard to impose the further condition that the \emph{incidence graph} is $K_{s,s}$-free, where $K_{s, s}$ denotes the complete bipartite graph with vertex classes of size $s$. Here, we think of $s$ as a large constant, which may depend on $d$, but no other parameters. In geometric terms, this condition says that there do not exist $s$ hyperplanes whose intersection contains more than $s$ points.

Given a set of points $\cP$ and set of geometric objects (e.g. hyperplanes) $\cH$, the \emph{incidence graph} of $(\cP,\cH)$ is the bipartite graph $G(\cP,\cH)$ with vertex classes $\cP$ and $\cH$, and $\{x,H\}\subset \cP\cup \mathcal{H}$ is an edge if $x\in H$.  We denote by $I(\cP,\cH)$ the number of edges of the incidence graph, i.e. $I(\cP,\cH)$ is the number of incidences between $\cP$ and $\cH$. 

Beyond point-hyperplane incidences, we obtain sharp bounds for several other well-studied incidence problems. We find the maximum number of incidences between points and $d$-dimensional varieties of bounded degree over arbitrary fields assuming the incidence graph is $K_{s,s}$-free. Then, we determine the maximum number of unit distances between $n$ points in $\bF^d$, assuming the unit distance graph is $K_{s, s}$-free. Finally, we conclude with Zarankiewicz's problem for so called algebraic graphs.

\subsection{Point-hyperplane incidences}

The problem of bounding the number of edges in $K_{s,s}$-free point-hyperplane incidence graphs has a long history, with the initial motivation coming from computational geometry. How compactly can one record the incidences in a configuration of points and hyperplanes? Instead of writing down whether each of the $mn$ pairs $(x, H)\in \cP\times \cH$ forms an incidence, one can compress the incidence graph by representing it as the union of complete bipartite subgraphs. This cannot be done very efficiently in general, for example  if $G=G_{n,1/2}$ is the Erd\H{o}s-R\'enyi random graph, one needs a list of $\Omega(\frac{n^2}{\log n})$ vertices to encode any partition. However, for incidence graphs Chazelle \cite{Ch93} developed a space partitioning technique to show that this can be improved to $O_{d, s}\big(n^{2-\frac{2}{d+1}} \log n\big)$ when $m=n$. Another motivation comes from the counting version of Hopcroft's problem, which asks, given $m$ points and $n$ hyperplanes in $\bR^d$, how fast can one determine the number of incidences. The work of Chazelle \cite{Ch93} achieved the first subquadratic algorithm for this problem and Erickson \cite{Erickson} gave a lower bound for the running time of algorithms classified as partitioning algorithms. He showed that particularly hard instances for determining the number of incidences are configurations $(\cP, \cH)$ without a compact representation, i.e. a configuration without a large complete bipartite graph in the incidence graph.

Improving the works of Chazelle \cite{Ch93} and  Brass and Knauer \cite{BK03}, Apfelbaum and Sharir \cite{AS07} proved that if $\cP$ is a set of $m$ points and $\cH$ is a set of $n$ hyperplanes in $\mathbb{R}^d$, whose incidence graph contains no $K_{s, s}$, then 
\begin{equation}\label{equ:1}
    I(\cP, \cH)\leq O_{d,s}\left((mn)^{1-\frac{1}{d+1}}+m+n\right).
\end{equation}
This upper bound has not been improved in the past twenty years, however, matching lower bounds are only known for $d=2$, which coincides with the Szemer\'edi-Trotter theorem. The common feature in all of the proofs of (\ref{equ:1}) and related results is that they rely on certain space partitioning results such as cuttings \cite{Ch93, Clarkson} or polynomial partitioning \cite{GK, SolymosiTao}. These techniques are highly geometric and have no analogues over finite fields.

The currently best known lower bound for $d\geq 3$ was recently achieved by Sudakov and Tomon \cite{ST23}, improving on the constructions of Brass and Knauer \cite{BK03} and Balko, Cibulka, and Valtr \cite{BCV}: if $s$ is sufficiently large with respect to $d$, then there exist a set of $m$ points $\cP$ and a set of $n$ hyperplanes $\cH$ in $\bR^d$ whose incidence graph is $K_{s,s}$-free and 
\[I(\cP, \cH)\geq \begin{cases}
    \Omega_d\left((mn)^{1-(2d+3)/{(d+2)(d+3)}}\right) &\text{  if $d$ is odd,}\\
    \Omega_d\left((mn)^{1-(2d^2+d-2)/{(d+2)(d^2+2d-2)}}\right) &\text{  if $d$ is even.}
\end{cases}\]

A natural variant of this problem asks about the maximum number of incidences between points and hyperplanes in a $d$-dimensional space $\mathbb{F}^d$ for a general field $\mathbb{F}$. If $d=2$ and $p$ is a prime, taking every point of $\mathbb{F}_p^2$ and $n<p^2$ arbitrary lines, we get $m=p^2$ points, and $np=\Theta(nm^{1/2})$ incidences, beating the Szemer\'edi-Trotter bound as long as $n\gg p$. On the other hand, a simple application of the K\H{o}v\'ari-S\'os-Tur\'an theorem \cite{KST} shows that one cannot have more than $O(nm^{1/2}+m)$ incidences over any field, by noting that the incidence graph of points and lines is always $K_{2,2}$-free. The work of Bourgain, Katz, and Tao \cite{BKT} shows that better upper bounds can be obtained over $\mathbb{F}_p$ assuming $m,n\leq p^{2-\delta}$.

However, point-hyperplane incidence bounds in higher dimensions turn out to be more elusive. One such result in three dimensions was obtained by Rudnev \cite{Rudnev} (see also de Zeeuw \cite{Z} for a shorter proof). He proved that in $\mathbb{F}_p^3$, if no $s$ planes contain the same line, then there are at most $O(m\sqrt{n}+ms)$ incidences, assuming $n\leq p^2$. This point-plane incidence bound in $\bF_p^3$ was later used to show various improved sum-product estimates over finite fields in \cite{AMRS,MS21,RRS16}. For $d>3$, barely anything is known about incidences of points and hyperplanes in $\bF_p^d$.

In this paper, we propose a novel approach to derive bounds on the maximal number of incidences between $m$ points and $n$ hyperplanes in a $d$-dimensional vector space $\mathbb{F}^d$, assuming the incidence graph is $K_{s,s}$-free. We combine new graph-theoretic techniques with probabilistic arguments and prove bounds which are sharp for the whole range of parameters $m$ and $n$ with a suitable choice of field. Surprisingly, for many interesting pairs of values $(m,n)$, our upper bound matches the upper bound (\ref{equ:1}). This is fairly unexpected, as all proofs of (\ref{equ:1}) rely on highly geometric techniques, while we employ only combinatorial ideas. To give a snippet of our most general result, we prove the following in the special case $m=n$.

\begin{theorem}\label{thm:upper bounds K_s,s symmetric}
Let $d, s$ be positive integers and let $\bF$ be a field. If $\cP$ is a set of $n$ points and $\cH$ is a set of $n$ hyperplanes in $\bF^d$ such that no $s$ points lie on the intersection of $s$ hyperplanes, then $$I(\cP, \cH)\leq O\left((s+d^3)n^{2-1/\lceil \frac{d+1}{2}\rceil}\right).$$
\end{theorem}

\noindent
Note that this bound matches the best known bound (\ref{equ:1}) when $d$ is odd. Before stating the general bound, which depends on the relationship of $m$ and $n$, let us introduce a piece of notation. Namely, we define the quantities $\alpha_t=\frac{t}{d+2-t}$ for $t\in \{2, \dots, d\}$ and $\beta_t=\frac{t}{d+1-t}$ for $t\in \{1, \dots, d\}$. 
Now we can state our main theorem in its full generality.

\begin{theorem}\label{thm:main thm K_s,s}
Let $d, s,m,n$ be positive integers and $\alpha>0$ such that $n=m^{\alpha}$, and let $\bF$ be a field. Let $\cP$ be a set of $m$ points and $\cH$ be a set of $n$ hyperplanes in $\bF^d$ such that the incidence graph of $(\cP, \cH)$ is $K_{s, s}$-free. Then
$$I(\cP, \cH)\leq \begin{cases} O_{s,d}(m)  &\mbox{ if } \alpha\in (0,\beta_1],\\
O_{s, d}(m^{1-\frac{1}{d+2-t}}n)&\mbox{ if }\alpha\in [\beta_{t-1},\alpha_t]\mbox{ for some }t\in \{2,\dots,d\},\\
O_{s, d}(mn^{1-\frac{1}{t}}) &\mbox{ if }\alpha\in [\alpha_t,\beta_t]\mbox{ for some }t\in \{2,\dots,d\},\\
O_{s,d}(n)  &\mbox{ if } \alpha\in [\beta_d,\infty).\end{cases}$$

Moreover, these bounds are tight, i.e. for every $d$, there exist $s=s(d)$ and $c=c(d)$ such that the following hold. For every $\alpha>0$ and sufficiently large integer $m$, there exist a field $\bF=\bF(d,m,\alpha)$,  a set of $m$ points $\cP$ and a set  of $n=\lfloor m^{\alpha}\rfloor$ hyperplanes $\cH$ in $\bF^d$ such that the incidence graph of $(\cP, \cH)$ is $K_{s, s}$-free and  
$$I(\cP, \cH)\geq \begin{cases} m  &\mbox{ if } \alpha\in (0,\beta_1],\\
cm^{1-\frac{1}{d+2-t}}n&\mbox{ if }\alpha\in [\beta_{t-1},\alpha_t]\mbox{ for some }t\in \{2,\dots,d\},\\
cmn^{1-\frac{1}{t}} &\mbox{ if }\alpha\in [\alpha_t,\beta_t]\mbox{ for some }t\in \{2,\dots,d\},\\
n&\mbox{ if } \alpha\in [\beta_d,\infty).\end{cases}$$
\end{theorem}

Let us compare our result with (\ref{equ:1}) in the non-trivial regime $\alpha\in [1/d, d]$. Interestingly, in case $\alpha=\beta_t=\frac{t}{d+1-t}$ for some integer $t\in \{1, \dots, d\}$, Theorem \ref{thm:main thm K_s,s} shows that the number of incidences is $O_{s,d}((mn)^{1-1/(d+1)})$, which exactly matches (\ref{equ:1}). The other extreme is when $\alpha=\alpha_t=\frac{t}{d+2-t}$ for some $t\in \{2, \dots, d\}$, in which case we get $O_{s,d}((mn)^{1-1/(d+2)})$. For other values of $\alpha$, our upper bound is $O_{s,d}((mn)^{1-\gamma})$ for some $\frac{1}{d+2}\leq \gamma\leq \frac{1}{d+1}$. 

The second half of the theorem shows that these upper bounds cannot be improved unless further assumptions on the field $\bF$ are made. To show this, we consider optimal constructions of so called subspace evasive sets \cite{DL12,ST23}. We follow the ideas of \cite{BCV,BK03,ST23} to transform large subspace evasive sets into point-hyperplane configuration with many incidences and no large bipartite subgraphs.

\bigskip

We also study another well-known problem about finding large complete bipartite graphs in dense incidence graphs. The motivation behind studying this problem is the intuitive understanding that a large number of incidences between points and hyperplanes is always explained by the existence of large complete bipartite subgraphs in the incidence graph. To make things more precise, given a set of points $\cP$ and set of hyperplanes $\cH$, we define $\rs(\cP,\cH)$ as the maximum number of edges in a complete bipartite subgraph of the incidence graph of $(\cP,\cH)$, i.e. the maximum of $r\cdot s$ over all $r$ and $s$ such that $K_{r,s}$ is a subgraph of the incidence graph. We consider configurations of $m$ points and $n$ hyperplanes with $\eps mn$ incidences.

In the real $d$-dimensional space, Apfelbaum and Sharir \cite{AS07} proved that $\rs(\cP,\cH)=\Omega_d(\eps^{d-1}mn)$ for  $\eps>\Omega(n^{-1/(d-1)})$. Also, if $\eps>\Omega((mn)^{-\frac{1}{d-1}})$, then there exist a set of points $\cP$ and set of hyperplanes $\cH$ in $\bR^d$ with $\rs(\cP, \cH)\leq O_d(\eps^{\frac{d+1}{2}}mn)$. Note that in case $d=3$, the lower and upper bounds match for $\eps>\Omega(n^{-1/2})$. Do \cite{Do20} improved the lower bound in $\bR^4$ and $\bR^5$ for a large range of $m, n, \eps$ to match the upper bound up to logarithmic factors. One might be also interested how the minimum of $\rs(\cP,\cH)$ depends on the dimension $d$ as well, which is a question motivated by certain variants of the celebrated log-rank conjecture of Lov\'asz and Saks \cite{logrank}. In this regime, it follows from  Fox, Pach, and Suk \cite{FPS16} that $\rs(\cP,\cH)>\eps^{d+1}2^{-O(d\log d)}mn$. For constant $\eps$, this is improved by a  recent result of  Singer and Sudan \cite{SS22}, who obtained an exponential dependence on $d$ in the lower bound $\rs(\cP,\cH)=\Omega(\eps^{2d}mn/d)$. We refer the interested reader to \cite{SS22}  about the relationship of this problem and the log-rank conjecture.

Here, we extend and improve some of these results by considering arbitrary fields, and obtain tight bounds in all cases. Following Apfelbaum and Sharir \cite{AS07}, we  define $\rs_d(n, m, I)=\min \rs(\cP, \cH)$, where the minimum is taken over all fields $\bF$, all sets  $\cP\subset \bF^d$ of at most $m$ points, all sets $\cH$ of at most $n$ hyperplanes in $\bF^d$, that satisfy $I(\cP, \cH)\geq I$. Note that requiring that $\cP$ is a set of at most $m$ points instead of exactly $m$ points makes no crucial difference, since one can always add points to the set $\cP$ which are incident to no hyperplanes (if $\bF$ is a finite field, one might need to pass to an extension of $\bF$). A similar remark holds for hyperplanes. 

\begin{theorem}\label{thm:main thm rs}
Let $d$ and $m\leq n$ be positive integers, $\eps\in (0,1)$ and $I=\eps mn$. Then
$$\rs_d(m,n,I), \rs_d(n,m,I)=\begin{cases}
\Theta_d(\eps^{d-1}mn) &\mbox{ if } \eps>100\max\{m^{-\frac{1}{d-1}}, n^{-\frac{1}{d}}\},\\
\Theta_d(\eps n) &\mbox{ if } \eps< \frac{1}{4}\max\{m^{-\frac{1}{d-1}}, n^{-\frac{1}{d}}\}.
\end{cases}$$
\end{theorem}

\noindent
In the regime $\eps\gg \max\{m^{-\frac{1}{d-1}}, n^{-\frac{1}{d}}\}$, the implied constant of the lower bound is only exponential in $d$. Therefore, our theorem matches the dependence on $\eps$ coming from the result of Apfelbaum and Sharir \cite{AS07}, as well as the exponential dependence on the dimension obtained by Sudan and Singer \cite{SS22}.

On the other hand, in case $\eps\ll \max\{m^{-\frac{1}{d-1}}, n^{-\frac{1}{d}}\}$, Theorem~\ref{thm:main thm rs} tells us that the largest complete bipartite subgraph might be just a star. Finally, observe that in case $m=n$, there is a large jump around $\varepsilon\approx n^{-1/d}$, where $\rs_d(n, n, I)$ suddenly jumps from $\approx n^{1-1/d}$ to $\approx n^{1+1/d}$.

\subsection{Point-variety incidences}

The number of incidences between points and varieties in the real space has been extensively studied (see Section~\ref{sec:alg geo background} for a gentle introduction on varieties). Pach and Sharir \cite{PS1,PS2} proved that if $\cP$ is a set of $m$ points and $\cV$ is a set of $n$ algebraic curves of degree at most $\Delta$ in $\mathbb{R}^2$ such that $G(\cP,\cV)$ is $K_{s,t}$-free, then $$I(\cP,\cV)=O_{s,t,\Delta}\left(m^{\frac{s}{2s-1}}n^{\frac{2s-2}{2s-1}}+m+n\right).$$
Zahl \cite{Zahl} and Basu and Sombra \cite{BS} extended this bound to dimension 3 and 4, respectively. Subsequently, Fox, Pach, Sheffer, Suk, and Zahl \cite{FPSSZ} established the following general bound in $\mathbb{R}^D$: if $\cP$ is a set of $m$ points and $\cV$ is a set of $n$ varieties in $\mathbb{R}^{D}$ of degree at most $\Delta$, and $G(\cP,\cV)$ is $K_{s,t}$-free, then
\begin{equation}\label{equ:FPSSZ}
    I(\cP,\cV)=O_{D, s,t,\Delta,\varepsilon}\left(m^{\frac{(D-1)s}{Ds-1}+\varepsilon}n^{\frac{D(s-1)}{Ds-1}}+m+n\right),
\end{equation}
for every $\varepsilon>0$. However, this bound is only known to be tight in the special case when $s=2$ and $t$ is sufficiently large, see \cite{Sheffer}.

Here, we consider a similar problem over arbitrary fields. That is, we consider the problem of bounding the number of incidences between points and varieties in $\bF^D$, where $\bF$ is an arbitrary field, under the assumption that the incidence graph is $K_{s, s}$-free (our proofs and results remain essentially the same if we consider the more general problem of forbidding $K_{s,t}$ for $t\geq s$, so we assume $s=t$ to simplify notation).

\begin{theorem}\label{thm:varieties upper}
Let $\cP$ be a set of $m$ points and let $\cV$ be a set of $n$ varieties in $\bF^D$, each of dimension $d$ and degree at most $\Delta$. If the incidence graph $G(\cP, \cV)$ is $K_{s, s}$-free, then \[I(\cP, \cV)\leq O_{d,\Delta, s}(m^{\frac{d}{d+1}} n+m).\]
\end{theorem}

Let us compare our bound with (\ref{equ:FPSSZ}). As $s\rightarrow\infty$, the bound in (\ref{equ:FPSSZ}) approaches the shape $m^{\frac{D-1}{D}}n+m+n$, which matches the one in Theorem \ref{thm:varieties upper} in case $d=D-1$. For smaller $d$, Theorem~\ref{thm:varieties upper} provides even better bounds. Curiously, our bound does not depend on the dimension $D$ of the ambient space, it only depends on the degree and the dimension of the varieties. Moreover, our bounds are essentially tight, as the next theorem shows.

\begin{theorem}\label{thm:construction varieties}
Let $d<D$ be positive integers, $\alpha>0$, and $m, n$ be positive integers such that $n=\lfloor m^{\alpha}\rfloor$, and $m,n$ are sufficiently large with respect to $D$ and $\alpha$. Then there exist a prime $p$, a set of $m$ points $\cP\subseteq \bF_p^D$ and a set $\cV$ of $n$ varieties in $\bF_p^D$ of dimension $d$ and degree at most $\Delta=\lceil(1+\alpha)(d+1)\rceil^2$ such that $G(\cP, \cV)$ does not contain $K_{s, s}$ with $s=\Delta^{1/2}$, and $$I(\cP, \cV)\geq \Omega_{d,\alpha}(m^{\frac{d}{d+1}} n).$$
\end{theorem}

We provide two proofs of Theorem~\ref{thm:varieties upper}. The first one uses the framework of induced Tur\'an problems. We construct a bipartite graph $H_{d,\Delta}$ such that incidence graphs of points and varieties in $\bF^D$ of dimension $d$ and degree at most $\Delta$ avoid $H_{d,\Delta}$ as an induced subgraph, and every vertex of $H_{d,\Delta}$ in one of the parts has degree at most $d+1$. Then, we adapt the ideas of \cite{HMST} to prove bounds on the number of edges of induced $H_{d,\Delta}$-free and $K_{s,s}$-free bipartite graphs. 

An alternative proof of Theorem~\ref{thm:varieties upper} is obtained by establishing a bound on the VC-dimension of incidence graphs and the rate of growth of their shatter functions. To do this we bound the number of incidence-patterns between points and lower-dimensional varieties, extending
the celebrated result of R\'onyai, Babai and Ganapathy \cite{RBG} on the number of zero-patterns of polynomials.

We remark that graphs with bounded VC-dimension avoid induced copies of a certain forbidden bipartite subgraph with similar properties as $H_{d,\Delta}$, but our explicit construction of $H_{d,\Delta}$ is more straightforward. Moreover, in certain situations, like the unit-distance problem for instance, finding specific forbidden bipartite graphs leads to better bounds compared to the ones which follow from VC-dimension arguments.

Finally, in order to prove Theorem~\ref{thm:construction varieties} as well as other constructions in this paper, we use the results from \cite{ST23}, which are based on the random algebraic method originally introduced by \cite{Bukh}.

\subsection{Unit distances in finite fields}

The celebrated Erd\H{o}s unit distance problem \cite{Erdos1,Erdos2} is one of the most notorious open problems in combinatorial geometry, asking to estimate the function $f_d(n)$, the maximum number of unit distances spanned by $n$ points in $\mathbb{R}^d$. The cases $d=2,3$ are the most difficult ones, with a large gap remaining between the best known lower and upper bounds. For $d=2$, the state-of-the-art is $n^{1+\Omega(1/\log\log n)}<f_2(n)<O(n^{4/3})$, where the lower bound is due to Erd\H{o}s \cite{Erdos2}, and the upper bound is due to Spencer, Szemer\'edi and Trotter \cite{SST}. For $d=3$, we know that $\Omega(n^{4/3}\log\log n)<f_3(n)<n^{1.498}$ by \cite{Erdos2} and \cite{Zahl2}, respectively. On the other hand, for $d\geq 4$, it is not difficult to construct a set of points achieving $f_d(n)=\Theta_d(n^2)$ \cite{Lenz}. Indeed, if $d\geq 4$, one can take two orthogonal linear subspaces $V_1$ and $V_2$, each of dimension at least 2. Then, placing $n/2$ points on the origin-centered sphere $S_i\subset V_i$ of radius $1/\sqrt{2}$ for $i=1,2$, we get at least $n^2/4$ unit distances. However, as it is shown in \cite{FPSSZ}, in some sense this construction is the only way one can get a quadratic number of unit distances. More precisely, if we assume that no $s$ points are contained in a $(d-3)$-dimensional sphere, then the maximum number of unit distances is at most 
$$O_{d,s,\varepsilon}(n^{2-\frac{2}{d+1}+\varepsilon})$$
for every $\varepsilon>0$ (see also \cite{FK} for further strengthening). The key observation is that such unit distance graphs are semi-algebraic of complexity at most $2$ containing no copy of $K_{s,s}$. The bounds on the number of edges of such graphs follow from the more general bounds (\ref{equ:semi-algebraic}) on Zarankiewicz's problem for algebraic graphs, which will be discussed in the next section. In the context of finite fields, this problem also received attention when $d=3$. In this dimension, there are stronger natural restrictions on the unit-distance graphs, beyond being $K_{s, s}$-free. Under these conditions, Zahl \cite{Zahl3} proved that a set of $n$ points determines at most $O(n^{3/2})$ unit distances assuming $-1$ is not a square in $\bF$, while Rudnev \cite{Rudnev2} established the upper bound $O(n^{8/5})$ without this assumption. 

Here, we completely resolve the problem of bounding the number of unit distances over arbitrary fields, assuming the unit distance graph is $K_{s,s}$-free. A \textit{unit sphere} in $\bF^d$ with center $a=(a_1, \dots, a_d)\in \bF^d$ is defined as the set of points $(x_1, \dots, x_d)\in \bF^d$ which satisfy $$(x_1-a_1)^2+\dots+(x_d-a_d)^2=1.$$
Furthermore, points $a$ and $b$ are at \emph{unit distance} if $b$ is contained in the unit sphere with center $a$.

\begin{theorem}\label{thm:unit_distance}
Let $\cP$ be a set of $n$ points in $\bF^d$. If the unit distance graph of $\cP$ does not contain $K_{s, s}$, then the number of unit distances spanned by $\cP$ is at most $O_{d,s}(n^{2-\frac{1}{\lceil d/2\rceil +1}})$.
\end{theorem}

Furthermore, this bound is sharp in the following sense.

\begin{theorem}\label{thm:sphere_construction}
Let $n, d$ be positive integers. There exists a constant $s=s(d)$, a finite field $\bF_q$ and a set of $n$ points $\cP\subseteq \bF_q^d$ such that the unit distance graph on $\cP$ does not contain $K_{s, s}$ and  $\cP$ spans $\Omega_d(n^{2-\frac{1}{\lceil d/2\rceil+1}})$ unit distances.
\end{theorem}

\subsection{Algebraic graphs}
 Zarankiewicz's problem asks for the maximum number of edges in an $n$-vertex graph without a copy of $K_{s, s}$. Fox, Pach, Sheffer, Suk, and Zahl \cite{FPSSZ} considered a variant of this problem in which case the host graph is semi-algebraic. We say that a bipartite graph $G=(A,B,E)$ is \emph{semi-algebraic} of description complexity $t$ in $(\mathbb{R}^{d_1},\mathbb{R}^{d_2})$, if $A\subset \mathbb{R}^{d_1}$, $B \subset \mathbb{R}^{d_2}$, and the edges are defined by the sign-patterns of $t$ polynomials $f_1,\dots,f_t:\mathbb{R}^{d_1}\times \mathbb{R}^{d_2}\rightarrow \mathbb{R}$, each of degree at most $t$. In \cite{FPSSZ} it is proved that if such a graph $G$ is $K_{s,s}$-free with $|A|=m$ and $|B|=n$, then
 $$|E(G)|=O_{d_1,d_2,t,s,\varepsilon}\left( m^{\frac{d_2(d_1-1)}{d_1d_2-1}+\varepsilon}n^{\frac{d_1(d_2-1)}{d_1d_2-1}}+m+n\right)$$
for every $\varepsilon>0$. In the special case of interest $d=d_1=d_2$, this gives
\begin{equation}\label{equ:semi-algebraic}
    |E(G)|=O_{d,t,s,\varepsilon}\left((mn)^{\frac{d}{d+1}+\varepsilon}+m+n\right).
\end{equation}

This bound is only known to be tight for $d=2$. It is worth highlighting that this upper bound matches (\ref{equ:1}) up to the $\varepsilon$ error term.
 
We consider the analogous problem for \emph{algebraic graphs}, that is, graphs that are defined with respect to zero-patterns of polynomials over arbitrary fields. Given two sets of points $\cP\subseteq \bF^{d_1}, \cQ\subseteq \bF^{d_2}$, a collection of $t$ polynomials $f_1, \dots, f_t:\bF^{d_1}\times \bF^{d_2}\to \bF$, and a boolean formula $\Phi$, we define the bipartite graph $G$ on the vertex set $\cP\cup \cQ$, where $x\in \cP$ and $y\in \cQ$ are adjacent if and only if 
\[\Phi([f_1(x, y)=0], \dots, [f_t(x, y)=0])=1.\]
We say that $G$ is an \emph{algebraic graph of description complexity $t$} if it can be described by $t$ polynomials $f_1, \dots, f_t$, each of which has degree at most $t$. 

\begin{theorem}\label{thm:algebraic zarankiewicz}
Let $G$ be an algebraic graph of description complexity at most $t$ on the vertex set $\cP\cup \cQ$, where $\cP\subseteq \bF^{d_1}$, $\cQ\subseteq \bF^{d_2}$ and $|\cP|=m$, $|\cQ|=n$. If $G$ is $K_{s, s}$-free, the number of edges in $G$ is at most $$O_{d_1,d_2,t, s}(\min\{m^{1-1/d_1}n, mn^{1-1/d_2}\}).$$
Moreover, this bound is tight, for every $d_1,d_2$, there exists $t$ such that if $m, n$ are sufficiently large and $s\geq d_1+d_2$, the following holds. There exist a field $\bF$ and a $K_{s, s}$-free algebraic graph $G$ of description complexity at most $t$ in $(\bF^{d_1}, \bF^{d_2})$ with parts of size $m, n$ and at least $\Omega(\min\{m^{1-1/d_1}n, mn^{1-1/d_2}\})$ edges.
\end{theorem}

\noindent
\textbf{Paper organization:} In Section~\ref{sec:illustration}, we give a short illustration of our techniques discussing point-hyperplane incidences in 3-dimensional spaces, and introduce the key lemma connecting combinatorics and geometry. Then, in Section~\ref{sec:point-hyperplane incidences}, we prove our general results on point-hyperplane incidences, Theorems~\ref{thm:main thm K_s,s} and Theorem~\ref{thm:main thm rs}, along with the constructions showing the tightness of these bounds. In Section~\ref{sec:point-variety incidences} we discuss point-variety incidences and show two proofs of Theorem~\ref{thm:varieties upper}, while postponing the constructions until Section~\ref{sec:algebraic zarankiewicz}. The unit distance problem is discussed in Section~\ref{sec:unit distances}, where the proofs of Theorems~\ref{thm:unit_distance} and \ref{thm:sphere_construction} are presented. In Section~\ref{sec:algebraic zarankiewicz} we discuss the Zarankiewcz problem for algebraic graphs and show Theorems~\ref{thm:construction varieties} and~\ref{thm:algebraic zarankiewicz}. We conclude our paper with a few remarks.

\section{Forbidden induced patterns and illustration in dimension $3$}\label{sec:illustration}

The key idea used to prove our bounds is that incidence graphs of points and hyperplanes in $\bF^d$ avoid a simple family $\mathcal{F}_d$ of graphs as induced subgraphs. The family $\mathcal{F}_d$ includes all bipartite graphs on vertex classes $\{a_1,\dots,a_d\}$ and $\{b_1,\dots,b_d\}$, where $a_ib_j$ is an edge for $i\geq j-1$, and $a_ib_{i+2}$ is a non-edge for $i=1,\dots,d-2$. 

An alternative and perhaps easier way to think about the family is through the notion of a \textit{pattern}. We define a \textit{pattern} to be an edge labeling of a complete bipartite graph, in which every edge receives one of the labels $0$, $1$ or $*$. Every such pattern $\Pi$ defines a family of graphs $\cF_\Pi$  on the same vertex set, where a graph $F$ is in the family if all edges of the pattern labeled by $0$ do not appear in $F$, all edges labeled by $1$ appear in $F$, while the edges labeled by $*$ may or may not appear. To simplify the terminology, we say that a graph $G$ contains a pattern $\Pi$ if it contains an induced copy of some graph in the family $\cF_\Pi$.

The family $\cF_d$ discussed above can be constructed in this way from the following pattern $\Pi_d$. Namely, $\Pi_d$ is a pattern on a balanced bipartite graph on vertices $a_1, \dots, a_d$, $b_1, \dots, b_d$ where $a_ib_{j}$ is labeled by $1$ for $i\geq j-1$ and $a_ib_{i+2}$ is labeled by $0$, while all other edges are labeled by $*$. See Figure \ref{fig1} for an illustration.

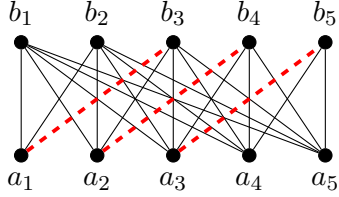
\begin{figure}[ht]
    \begin{center}
	\begin{tikzpicture}
            \node[vertex,label=below:$a_1$] (a1) at (-2,-1) {};
            \node[vertex,label=below:$a_2$] (a2) at (-1,-1) {};
            \node[vertex,label=below:$a_3$] (a3) at (0,-1) {};
            \node[vertex,label=below:$a_4$] (a4) at (1,-1) {};
            \node[vertex,label=below:$a_5$] (a5) at (2,-1) {};

             \node[vertex,label=above:$b_1$] (b1) at (-2,0.5) {};
            \node[vertex,label=above:$b_2$] (b2) at (-1,0.5) {};
            \node[vertex,label=above:$b_3$] (b3) at (0,0.5) {};
            \node[vertex,label=above:$b_4$] (b4) at (1,0.5) {};
            \node[vertex,label=above:$b_5$] (b5) at (2,0.5) {};

            \draw (a5) edge (b5);
            \draw (a5) edge (b4);
            \draw (a5) edge (b3);
            \draw (a5) edge (b2);
            \draw (a5) edge (b1);

            \draw (a4) edge (b5);
            \draw (a4) edge (b4);
            \draw (a4) edge (b3);
            \draw (a4) edge (b2);
            \draw (a4) edge (b1);

            \draw (a3) edge (b4);
            \draw (a3) edge (b3);
            \draw (a3) edge (b2);
            \draw (a3) edge (b1);

            \draw (a2) edge (b3);
            \draw (a2) edge (b2);
            \draw (a2) edge (b1);

            \draw (a1) edge (b2);
            \draw (a1) edge (b1);

            \draw[line width=0.5mm, red, dashed] (a3) -- (b5);
            \draw[line width=0.5mm, red, dashed] (a2) -- (b4);
            \draw[line width=0.5mm, red, dashed] (a1) -- (b3);
        \end{tikzpicture}
    \end{center}
    \caption{The pattern $\Pi_5$, where black edges are labeled 1, and the red dashed edges are labeled 0.}
    \label{fig1}
\end{figure}

The following lemma is the main geometric ingredient that we use, which allows us to reduce the incidence bounds to a problem in extremal graph theory.

\begin{lemma}\label{lemma:forbiddengraphs}
Let $d\geq 3$ be an integer and let $\cP$ be a set of distinct points and $\cH$ be a set of distinct hyperplanes in $\bF^d$. Then, the incidence graph $G(\cP, \cH)$ does not contain the pattern $\Pi_d$.
\end{lemma}
\begin{proof}
Assume, for the sake of contradiction, that $G=G(\cP, \cH)$ contains $\Pi_d$. Since the pattern $\Pi_d$ is symmetric, we may assume that the vertices $a_1, \dots, a_d$ correspond to points $x_1, \dots, x_d\in \cP$, and the vertices $b_1, \dots, b_d$ correspond to hyperplanes $H_1, \dots, H_d\in \cH$. We prove by induction that $\dim \bigcap_{i=1}^k H_i= d-k$ for all $2\leq k\leq d$. This suffices to derive the contradiction by observing that $x_{d-1}, x_d\in \bigcap_{i=1}^d H_i$, which is at most a $0$-dimensional space, i.e. a single point. This contradicts the assumption that $x_{1}, \dots, x_d$ are distinct.

To show that $\dim \bigcap_{i=1}^k H_i = d-k$, we begin by observing that $H_1\cap H_2$ must have dimension $d-2$, since $H_1, H_2$ are distinct hyperplanes with non-empty intersection. For any $k\geq 3$, the induction hypothesis implies that $ \bigcap_{i=1}^{k-1} H_i$ is an affine space of dimension $d-k+1$, which is not a subspace of $H_k$ since $x_{k-2}\in \bigcap_{i=1}^{k-1} H_i$ and $x_{k-2}\notin H_k$. Thus, $H_k$ must intersect $\bigcap_{i=1}^{k-1} H_i$ in an affine subspace of dimension $\dim \bigcap_{i=1}^{k-1} H_i-1=d-k$. This completes the proof.
\end{proof}

For a linear algebraic perspective on the pattern $\Pi_d$, one might consider the $(d-2)\times (d-2)$ matrix whose rows and columns are indexed by $a_1, \dots, a_{d-2}$ and $b_3, \dots, b_d$, where the entry corresponding to $a_i$ and $b_j$ equals $1$ if the edge $a_ib_j$ is labeled $0$ and equals $0$ if the edge $a_ib_j$ is labeled $1$ (entries corresponding to the edges labeled by $*$ may be arbitrary). This defines an upper-triangular matrix with non-zero diagonal entries, a condition which ensures full rank. 
A similar family of graphs was also studied recently by Sudakov and Tomon \cite{ST_Ramsey} to establish Ramsey theoretic properties of algebraically defined graphs. Using the above lemma, we can just focus on the hereditary family of graphs that avoid the pattern $\Pi_d$. Then, we use the dependent random choice technique \cite{FS} and combinatorial methods to show that if $G$ is a $\Pi_d$-free graph with sufficiently many edges, then $G$ must contain a large complete bipartite graph.

\medskip

Let us now illustrate how Lemma~\ref{lemma:forbiddengraphs} can be combined with the technique of dependent random choice to give a very short proof of Theorem~\ref{thm:upper bounds K_s,s symmetric} in dimension three. Note that when $d=3$, the forbidden pattern $\Pi_3$ is quite simple, since containing the pattern $\Pi_3$ is equivalent to containing the bipartite graph $K_{3,3}$ minus an edge as an induced subgraph.

Throughout this paper, we use standard graph theoretic notation. Given a graph $G$ and a vertex $v\in V(G)$, we denote by $\deg v$ the degree of $v$, and by $N(v)$ the neighbourhood of $v$. Also, if $v_1,\dots,v_t\in V(G)$, we denote by $N(v_1,\dots,v_t)$ the \emph{common neighbourhood} of $v_1,\dots,v_t$, i.e. $N(v_1,\dots,v_t)=N(v_1)\cap \dots\cap N(v_t)$.

\begin{lemma}\label{lemma:K_s,s for d=3}
Let $s\geq 2$ be an integer, let $\cP$ be a set of $n$ points and $\cH$ be a set of $n$ hyperplanes in $\bF^3$. If $G(\cP, \cH)$ does not contain $K_{s, s}$, then $I(\cP, \cH)\leq 2sn^{3/2}$. 
\end{lemma}
\begin{proof}
Throughout the proof, we denote the bipartition of the incidence graph $G=G(\cP, \cH)$ by $A\cup B$, where $A$ corresponds to the set of points and $B$ to the set of hyperplanes. This will help to separate the graph-theoretic arguments from the geometric ones.

We show that if $G=G(\cP, \cH)$ satisfies $e(G)\geq 2sn^{3/2}$ and does not contain $K_{s, s}$, then it must contain the pattern $\Pi_3$. To do this, we use the method of dependent random choice. More precisely, our goal is to find four vertices $v_1, v_2\in A, u_1, u_2\in B$ forming a copy of $K_{2,2}$ such that the common neighbourhoods $N(v_1)\cap N(v_2)$ and $N(u_1)\cap N(u_2)$ have size at least $s$. Since the graph $G$ does not contain $K_{s, s}$, there will be an edge missing between these two common neighbourhoods, which allows us to find the pattern $\Pi_3$ in $G$.

Let us begin by picking a random two vertex subset $\{v_1, v_2\}\subset A$ chosen from the uniform distribution on all $\binom{n}{2}$ pairs, and consider the common neighbourhood $X=N(v_1)\cap N(v_2)$ in $B$. The expected size of this common neighbourhood can be lower bounded by counting the number of three 
 element sets $\{v_1, v_2, u\}$ with $v_1, v_2\in A, u\in B$ and $v_1u, v_2u\in E(G)$ (such sets are sometimes also called as cherries). If the number of such sets is denoted by $T$, we have
\begin{align*}
\bE\Big[|X|\Big]&=\frac{T}{\binom{n}{2}}=\frac{1}{\binom{n}{2}}\sum_{u\in B}\binom{\deg u}{2}\geq \frac{2}{n-1}\binom{\frac{1}{n}\sum_{u\in B}\deg u}{2}\geq \frac{e(G)^2}{4n^3}\geq s^2,
\end{align*}
where the first inequality holds by the convexity of the function $\binom{x}{2}=\frac{x(x-1)}{2}$ for $x\geq 1$. On the other hand, let us denote by $Y$ the number of two element sets $\{u_1, u_2\}\subset X$ with common neighbourhood $N(u_1)\cap N(u_2)$ of size less than $s$. Then the expected value of $Y$ can be bounded as follows. For each two element set $\{u_1, u_2\}\subset B$, we have $u_1,u_2\in X$ if $v_1, v_2\in N(u_1)\cap N(u_2)$. Hence, if $|N(u_1)\cap N(u_2)|<s$, then the probability of $u_1,u_2\in X$ is less than $\binom{s}{2}/\binom{n}{2}<\left(\frac{s}{n}\right)^2$, so
\[\bE\left[Y\right]\leq \sum_{\{u_1, u_2\}\subset B}\left(\frac{s}{n}\right)^2\leq \binom{n}{2}\left(\frac{s}{n}\right)^2\leq \frac{s^2}{2}.\]

Since $\bE[|X|-Y]\geq s^2/2$, there exists a pair of vertices $v_1, v_2$ satisfying $|X|-Y\geq s^2/2$. This means that at least one pair $u_1, u_2\in X$ has a common neighbourhood of size at least $s$. Also, since $X\geq s^2/2\geq s$, the vertices $v_1, v_2$ have common neighbourhood of size at least $s$. As $G$ does not contain $K_{s, s}$, there must be vertices $v_3\in N(u_1)\cap N(u_2), u_3\in N(v_1)\cap N(v_2)$ for which $u_3v_3\notin E(G)$. But then $v_1, v_2, v_3, u_1, u_2, u_3$ form the pattern $\Pi_3$, which is a contradiction. This completes the proof.  
\end{proof}

\section{Point-hyperplane incidences}\label{sec:point-hyperplane incidences}

\subsection{Upper bounds for $K_{s, s}$-free incidence graphs}\label{sec:upper bounds K_{s, s}}

In this section, we prove the upper bounds in Theorem \ref{thm:main thm K_s,s}. We prepare the proof with the following lemma about finding the pattern $\Pi_d$ under some simple conditions. We recall that $\Pi_d$ is the pattern on a balanced bipartite graph on vertices $a_1, \dots, a_d$, $b_1, \dots, b_d$ where $a_ib_{j}$ is labeled by $1$ for $i\geq j-1$ and $a_ib_{i+2}$ is labeled by $0$, while all other edges are labeled by $*$.

\begin{lemma}\label{lemma:auxiliary upper bounds}
Let $s$ and $2\leq t\leq d$ be positive integers, and let $G=(A,B;E)$ be a $K_{s,s}$-free bipartite graph. Assume that $v_1, \dots, v_t\in A$ satisfy the following conditions:
\begin{itemize}
    \item $N(v_1, v_2)\supsetneq N(v_1, v_2, v_3)\supsetneq\cdots\supsetneq N(v_1, \dots, v_t)$,
    \item there exists a set $X\subseteq N(v_1, \dots, v_t)$ of size at least $d^{d+1}s$ such that every $(d+1-t)$-tuple of elements in $X$ has a common neighbourhood of size at least $s$.
\end{itemize}
Then $G$ contains the pattern $\Pi_d$.
\end{lemma}
\begin{proof}
The first condition implies that there exist vertices $u_1, \dots, u_{t-2}$ which satisfy $$u_i\in N(v_1, \dots, v_{i+1})\backslash N(v_{i+2}).$$ To find vertices $u_{t-1}, \dots, u_d$ and $v_{t+1}, \dots, v_d$, we define a $(d+2-t)$-uniform hypergraph $F$ on the set of vertices $X$. Let $\prec$ be an arbitrary total ordering on $X$, then the edges of the hypergraph $F$ will be those sets $\{u_{t-1}, \dots, u_d\}$ with $u_{t-1}\prec\dots\prec u_d$ for which there exists an index $i\in \{t-1,\dots,d-1\}$ such that $N(u_i, \dots, u_d)=N(u_{i+1}, \dots, u_d)$. We claim that $F$ has at most $$s (d+1-t) |X|^{d+1-t}$$ edges. To this end, note that there are $d+1-t$ choices for the index $i$ at which the equality $N(u_i, \dots, u_d)=N(u_{i+1}, \dots, u_d)$ can occur. Furthermore, for each of these choices, there are at most $|X|^{d-i}$ ways to choose $\{u_{i+1}, \dots, u_d\}$. Having fixed these vertices, which by assumption satisfy $|N(u_{i+1}, \dots, u_d)|\geq s$, we have less than $s$ vertices $u_i$ for which $N(u_i)\supseteq N(u_{i+1}, \dots, u_d)$, otherwise we have a $K_{s, s}$. Finally, there are at most $|X|^{i-(t-1)}$ ways to choose vertices $u_{t-1}, \dots, u_{i-1}$. Thus, the total number of edges of the hypergraph $F$ is at most 
\[|E(F)|< (d+1-t)|X|^{d-i}\cdot s\cdot |X|^{i-(t-1)}=s (d+1-t) |X|^{d+1-t},\]
as claimed. 

Observe that the inequality $\binom{|X|}{d+2-t}\geq \big(\frac{|X|}{d+2-t}\big)^{d+2-t}\geq s(d+1-t)|X|^{d+1-t}$ is satisfied by our assumption that $|X|\geq s d^{d+1}$. Hence, there exist $u_{t-1}\prec\dots\prec u_d$ in $X$ such that $\{u_{t-1},\dots,u_d\}$ is not an edge of $F$, or equivalently, $N(u_i, \dots, u_d)\subsetneq N(u_{i+1}, \dots, u_d)$ for all indices $t-1\leq i\leq d-1$. Picking vertices $$v_{i+2}\in N(u_{i+1}, \dots, u_d)\backslash N(u_i, \dots, u_d)$$ for all $t-1\leq i\leq d-1$, the vertices $v_1,\dots,v_d,u_1,\dots,u_d$ span a subgraph with pattern $\Pi_d$.
\end{proof}

The following is the main graph-theoretical lemma which allows us to prove the upper bounds on the number of incidences of points and hyperplanes. Let us recall that we have defined $\alpha_t=\frac{t}{d+2-t}$ for $t\in \{2, \dots, d\}$ and $\beta_t=\frac{t}{d+1-t}$ for $t\in \{1, \dots, d\}$.

\begin{lemma}\label{lemma:main upper bounds}
For integers $d,s\geq 2$, there exists $C=C(d, s)>0$ such that the following holds for every sufficiently large $m$. Let $G=(A,B;E)$ be a bipartite graph, $|A|=m, |B|=n$ and $n=m^\alpha$. If $G$ contains no $K_{s,s}$ as a subgraph and does not contain the pattern $\Pi_d$, then
$$e(G)\leq \begin{cases} Cm  &\mbox{ if } \alpha\in [0,\beta_1],\\
Cmn^{1-\frac{1}{t}} &\mbox{ if }\alpha\in [\alpha_t,\beta_t]\mbox{ for some }t\in \{2,\dots,d\},\\
Cm^{1-\frac{1}{d+2-t}}n&\mbox{ if }\alpha\in [\beta_{t-1},\alpha_t]\mbox{ for some }t\in \{2,\dots,d\},\\
Cn  &\mbox{ if } \alpha\in [\beta_d,\infty).\end{cases}$$
Moreover, if $\alpha=1$, i.e. when $m=n$, the constant $C=8(s+d^3)$ suffices.
\end{lemma}
\begin{proof}
Let us start by addressing the cases when $\alpha\in [\alpha_t, \beta_t]$ for some $t\in \{2,\dots,d\}$ or $\alpha\in [\beta_d,\infty)$. In the latter case, we set $t:=d$. Suppose that $e(G)\geq C(mn^{1-\frac{1}{t}}+n)$ for some large constant $C$, and observe that $mn^{1-\frac{1}{t}}\geq n$ if $\alpha\in [\alpha_t, \beta_t]$, and $n\geq mn^{1-\frac{1}{t}}$ if $\alpha\in [\beta_d,\infty)$. We show that there exists a $t$-tuple of vertices $v_1, \dots, v_t\in A$ which satisfies the assumptions of Lemma~\ref{lemma:auxiliary upper bounds}, which then implies that $G$ contains the pattern $\Pi_d$, contradiction. To do this, we use a variant of the dependent random choice method.

We call an ordered $t$-tuple of vertices $(v_1, \dots, v_t)$ \textit{bad} if $N(v_1, \dots, v_t)\geq s$ and $N(v_1, \dots, v_i)=N(v_1, \dots, v_{i+1})$ for some $i\in \{1,\dots,t-1\}$, and \emph{good} otherwise. We choose a random good $t$-tuple $(v_1, \dots, v_t)$ with the following sampling procedure. We sample the vertices $v_i$ in order, starting from $v_1$ which is a uniformly random element of $A$. At step $i$, having chosen vertices $v_1, \dots, v_{i-1}$, we have two cases. 
\begin{description}
\item[Case 1.] $|N(v_1, \dots, v_{i-1})|<s$.
In this case, any $t$-tuple containing $v_1, \dots, v_{i-1}$ is good, and then all subsequent vertices $v_j$ for $j\geq i$ can be sampled uniformly at random from $A\backslash\{v_1, \dots, v_{j-1}\}$.

\item[Case 2.] $|N(v_1, \dots, v_{i-1})|\geq s$.
Define the set $S_i=\{v\in A: N(v)\supseteq N(v_1, \dots, v_{i-1})\}$. Since $G(P, H)$ does not contain a copy of $K_{s, s}$, there are less than $s$ vertices in the set $S_i$. Sample $v_i$ uniformly at random from $A\backslash (S_i\cup\{v_1,\dots,v_{i-1}\})$.
\end{description}
Let us denote by $X$ the size of the common neighbourhood of $v_1, \dots, v_t$, and let $Y$ be the number of $(d+1-t)$-tuples of vertices $u_{t}, \dots, u_{d}\in N(v_1, \dots, v_t)$ which has less than $s$ common neighbours. If we show that $\bE[X-Y]\geq d^{d+1}s$, then one can pick a good $t$-tuple $v_1, \dots, v_t$ satisfying $X-Y\geq d^{d+1}s$ and delete one vertex from every $(d+1-t)$-tuple in $N(v_1, \dots, v_t)$ which has less than $s$ common neighbours. Then, one is left with the $t$-tuple $v_1, \dots, v_t$ and a subset of their common neighbourhood which satisfy all conditions of Lemma~\ref{lemma:auxiliary upper bounds}.

Thus, the main objective is to show that $\bE[X-Y]\geq d^{d+1}s$. By linearity of expectation, we have $\bE[X]=\sum_{u\in B}\Pb[u\in N(v_1, \dots, v_t)]$. Let us fix a vertex $u\in B$ with $\deg u\geq d+s$, and estimate $\Pb[u\in N(v_1, \dots, v_t)]$. We have 
\begin{align*}
\Pb[u\in N(v_1, \dots, v_t)]=\Pb[v_1, \dots, v_t\in N(u)]&=\prod_{i=1}^t\Pb\Big[v_i\in N(u)\Big|v_1, \dots, v_{i-1}\in N(u)\Big]\\
&\geq \prod_{i=1}^t \frac{\deg u -s-i+1}{m-i+1}>\left(\frac{\deg u -s-d}{m}\right)^t.
\end{align*}
Let $B_0\subseteq B$ be the set of vertices $u$ satisfying $\deg u\geq 2(d+s)$. Then, the number of edges incident to $B_0$ is $\sum_{u\in B_0}\deg u\geq C(mn^{1-1/t}+n)-2(d+s)n\geq e(G)/2$ for $C\geq 4(d+s)$. Therefore,
\begin{align*}
\bE[X]&\geq \sum_{u\in B_0} \left(\frac{\deg u -s-d}{m}\right)^t\geq \sum_{u\in B_0} \left(\frac{\deg u}{2m}\right)^t\\
&\geq |B_0|\left(\frac{\sum_{u\in B_0}\deg u}{2|B_0|m}\right)^t\geq n\left(\frac{e(G)}{4mn}\right)^t\geq \left(\frac{e(G)}{4mn^{1-1/t}}\right)^t,
\end{align*}
where the third inequality holds by convexity.
On the other hand, the probability that a given $(d+1-t)$-tuple of vertices $u_t, \dots, u_d\in B$ which has less than $s$ common neighbours lies in $N(v_1, \dots, v_t)$ can be bounded by
\begin{align*}
    \Pb[u_t, \dots, u_{d}\in N(v_1, \dots, v_t)]=\Pb[v_1, \dots, v_t\in N(u_t, \dots, u_{d})]\leq \left(\frac{|N(u_t, \dots, u_d)|}{m-s-d}\right)^t\leq (2s)^t m^{-t},
\end{align*}
assuming $m$ is sufficiently large with respect to $s$ and $d$. Thus, the expectation of $Y$ can be bounded as
\[\bE[Y]\leq \sum_{\substack{u_t, \dots, u_d\in B\\|N(u_t, \dots, u_d)|<s}} \left(\frac{2s}{m}\right)^t\leq n^{d+1-t}\left(\frac{2s}{m}\right)^t.
\]
When $\alpha\leq d=\beta_d$, we have $\alpha\leq \beta_t=\frac{t}{d+1-t}$. Hence, for a sufficiently large value of $C$, 
\[\bE[X-Y]\geq \left(\frac{e(G)}{4mn^{1-1/t}}\right)^t-n^{d+1-t}\left(\frac{2s}{m}\right)^t\geq (C/4)^t-(2s)^t m^{\alpha(d+1-t)-t}\geq (C/4)^t-(2s)^t\geq d^{d+1}s.\]
On the other hand, when $\alpha\geq d=\beta_d$, we have $$\bE[X-Y]\geq \Big(\frac{Cn}{4mn^{1-1/t}}\Big)^t-m\Big(\frac{2s}{m}\Big)^t=\big((C/4)^d-(2s)^d\big)\frac{n}{m^d}\geq (C/4)^d-(2s)^d\geq d^{d+1}s,$$ for a sufficiently large value of $C$. Thus, we indeed have $\bE[X-Y]\geq (C/4)^t-(2s)^t\geq d^{d+1}s$ in all cases and therefore this completes the first case of the proof.

Before we consider the rest of the cases, let us justify the assertion that when $\alpha=1$ one can take $C=8(s+d^3)$. Namely, if $\alpha=1$ we have $\alpha_t\leq \alpha\leq \beta_t$ for $t={\lceil{\frac{d+1}{2}}\rceil}$. Then, for the proof to go through, one needs to verify that the inequality $\bE[X-Y]\geq (C/4)^t-(2s)^t \geq sd^{d+1}$ holds with the proposed value of $C$, which is a consequence of a simple calculation.

Observe that the rest of the cases already follow by symmetry. Since $1/\alpha_t=\alpha_{d+2-t}$ and $1/\beta_{t-1}=\beta_{d+2-t}$, we have that if $\alpha\in [\beta_{t-1},\alpha_t]$ for some $t\in\{2,\dots,d\}$, then $1/\alpha\in [\alpha_{d+2-t},\beta_{d+2-t}]$. Also, if $\alpha\in [0,\beta_1]$, then $1/\alpha\in [\beta_d,\infty)$. Thus, if we reverse the roles of $m$ and $n$, together with the observation $m=n^{1/\alpha}$, we find ourselves in the setting of the cases already proved. This completes the proof of the theorem.
\end{proof}

Now we are ready to prove the upper bounds of Theorem~\ref{thm:upper bounds K_s,s symmetric} and Theorem~\ref{thm:main thm K_s,s}, which we restate as follows.

\begin{theorem}
Let $d, s\geq 2$ be integers, then there exists $C=C(d,s)>0$ such that the following holds for every $\alpha>0$ and field $\bF$. Let $\cP$ be a set of $m$ points and $\cH$ be a set of $n=m^\alpha$ hyperplanes in $\bF^d$ such that the incidence graph of $(\cP, \cH)$ is $K_{s, s}$-free. Then
$$I(\cP, \cH)\leq \begin{cases} Cm  &\mbox{ if } \alpha\in [0,\beta_1],\\
Cmn^{1-\frac{1}{t}} &\mbox{ if }\alpha\in [\alpha_t,\beta_t]\mbox{ for some }t\in \{2,\dots,d\},\\
Cm^{1-\frac{1}{d+2-t}}n&\mbox{ if }\alpha\in [\beta_{t-1},\alpha_t]\mbox{ for some }t\in \{2,\dots,d\},\\
Cn  &\mbox{ if } \alpha\in [\beta_d,\infty).\end{cases}$$
Moreover, if $\alpha=1$, the implied constant can be taken to be $C=8(s+d^3)$.
\end{theorem}

\begin{proof}
By Lemma~\ref{lemma:forbiddengraphs}, the graph $G(\cP, \cH)$ does not contain the pattern $\Pi_d$. Since the conclusion of Lemma~\ref{lemma:main upper bounds} does not hold, one of its assumptions must fail. Thus, we conclude that $G(\cP, \cH)$ either contains $K_{s, s}$ as a subgraph or $I(\cP, \cH)$ satisfies the required upper bound.
\end{proof}

\subsection{Large complete bipartite subgraph of dense incidence graphs}

The main goal of this section is to prove our lower bounds on $\rs_d(m, n, I)$ and $\rs_d(n, m, I)$ for $m\leq n$, that is, to prove Theorem~\ref{thm:main thm rs}. Recall that $\rs_d(n, m, I)=\min \rs(\cP, \cH)$, where the minimum is taken over all fields $\bF$, all sets  $\cP\subset \bF^d$ of at most $m$ points, all sets $\cH$ of at most $n$ hyperplanes in $\bF^d$, that satisfy $I(\cP, \cH)\geq I$.  As observed before, the lower bounds are only interesting in the regime $\eps>100 \max\{m^{-\frac{1}{d-1}}, n^{-\frac{1}{d}}\}$, otherwise taking the maximum degree vertex with its neighborhood gives a complete bipartite graph of the required size. Let us restate our theorem in this case.

\begin{theorem}\label{thm:main thm rs_upper}
Let $m\leq n$ be integers, $\eps>0$, and $I=\eps mn$. If $\eps>100 \max\{m^{-\frac{1}{d-1}}, n^{-\frac{1}{d}}\}$, then 
$$\rs_d(m, n, I),\rs_d(n, m, I)\geq \left(\frac{\eps}{100}\right)^{d-1} mn.$$
\end{theorem}

\noindent
Let us prepare the proof with some definitions. For a bipartite graph $G$, we denote by $\rs(G)$ the number of edges in the largest complete bipartite subgraph of $G=(A, B; E)$. Also, we say that a $t$-tuple of vertices $(v_1, \dots, v_t)\in A^t$ is a \textit{good $t$-tuple} if $N(v_1, v_2)\supsetneq N(v_1, v_2, v_3)\supsetneq \cdots \supsetneq N(v_1, \dots, v_t)$ and $|N(v_1, \dots, v_t)|\geq 2$. The reason this definition is useful is the following lemma, which allows us to embed the pattern $\Pi_d$ using good $d$-tuples. One can think of Lemma~\ref{lemma:good tuples implies Pi_d} as a simpler version of Lemma~\ref{lemma:auxiliary upper bounds} from Section~\ref{sec:upper bounds K_{s, s}}.

\begin{lemma}\label{lemma:good tuples implies Pi_d}
Suppose that $G$ contains a good $d$-tuple $(v_1, \dots, v_d)$. Then $G$ contains the pattern $\Pi_d$.
\end{lemma}
\begin{proof}
Picking $u_i\in N(v_1, \dots,  v_{i+1})\backslash N(v_1, \dots, v_{i+2})$ for $i\leq d-2$ and $u_{d-1}, u_d\in N(v_1, \dots, v_d)$, the vertices $v_1,\dots,v_d,u_1,\dots,u_d$ induce a subgraph with pattern $\Pi_d$.
\end{proof}

Let us reformulate Theorem \ref{thm:main thm rs_upper} in terms of good $d$-tuples.

\begin{lemma}\label{lemma:asymmetric upper bounds graphs eps}
Let $d\geq 2$ be a positive integer and let $G=(A, B; E)$ be a bipartite graph with at least $\eps mn$ edges, where $|A|=m$, $|B|=n$ and $n\geq m$. If $\eps\geq 100\max\{m^{-\frac{1}{d-1}}, n^{-\frac{1}{d}}\}$, then either $A$ contains a good $d$-tuple or $$\rs(G)\geq \left(\frac{\eps}{100}\right)^{d-1} mn.$$
\end{lemma}

We prove this lemma after some preparations. The following lemma shows how to build good $t$-tuples under the assumption that $G$ does not contain large complete bipartite subgraphs.

\begin{lemma}\label{lemma:one step embedding}
Let $\gamma,\eps>0$, and let $G=(A,B; E)$ be a bipartite graph such that $|A|=m$, $|B|=n$, and every vertex $u\in B$ has degree at least $\deg u\geq \frac{\eps}{3} n$. If $\rs(G)<\frac{\eps}{6} \gamma \cdot mn$, then for any set $X\subseteq B$ of size at least $\gamma n$, there exists a vertex $v\in A$ for which $|N(v)\cap X|\geq \frac{\eps}{6}|X|$ and $N(v)\cap X\neq X$. 
\end{lemma}
\begin{proof}
Let $X\subseteq B$ be a subset of size at least $\gamma n$. The idea of this proof is to double count the edges incident to $X$. Let $V_1$ denote the set of vertices $v\in A$ for which $X\subseteq N(v)$ and let $V_2$ denote the set of vertices $v\in A$ with $|N(v)\cap X|< \frac{\eps}{6}|X|$. To show the lemma, it suffices to find a vertex $v\in A$ belonging to neither $V_1$ nor $V_2$. 

Suppose, for the sake of contradiction, that $V_1\cup V_2=A$. Then, $e(A, X)=e(V_1, X)+e(V_2, X)$. Since the degree of every vertex in $B$ is at least $\frac{\eps m}{3}$, we have $e(A, X)\geq \frac{\eps}{3}m|X|$. On the other hand, since the graph induced on $V_1\cup X$ is complete bipartite, we have $e(V_1, X)\leq \rs(G)\leq \frac{\eps}{6}\gamma\cdot mn$. Also, by the definition of $V_2$, we have $e(V_2, X)=\sum_{v\in V_2} |N(v)\cap X|<|V_2|\cdot \frac{\eps}{6}|X|\leq \frac{\eps}{6}m|X|$. Hence,
$$\frac{\eps}{6}\gamma\cdot mn+\frac{\eps}{6}m|X|> \frac{\eps}{3} m|X|.$$
Rearranging this inequality gives $|X|< \gamma n$, which is a contradiction to our initial assumption.
\end{proof}

Now, we explain how to build good $(d-1)$-tuples in the graph $G$.

\begin{lemma}\label{lemma:good d-1 tuples}
Let $d\geq 3$ be a positive integer and let $G=(A,B; E)$ be a bipartite graph with $\eps mn$ edges, where $|A|=m$, $|B|=n$. If $\rs(G)\leq (\frac{\eps}{6})^{d-1}mn$, then $G$ contains a good $(d-1)$-tuple $(v_1, \dots, v_{d-1})\in A^{d-1}$ with common neighbourhood of size at least $|N(v_1, \dots, v_{d-1})|\geq 2\left(\frac{\eps}{6}\right)^{d-1} n$.
\end{lemma}
\begin{proof}
For $t=2,\dots,d-1$, we prove that there is a good $t$-tuple $(v_1,\dots,v_t)\in A^t$ with $|N(v_1, \dots, v_{t})|\geq 2\left(\frac{\eps}{6}\right)^{t} n$. We prove this by induction on $t$. For $t=2$, we simply need to find a pair of vertices $v_1, v_2\in A$ whose common neighbourhood has size at least $2\left(\frac{\eps}{6}\right)^{2} n$. Since the average degree of vertices in $B$ is at least $\eps m$, the average size of the common neighbourhood of two vertices of $A$ is $n\binom{\eps m}{2}/\binom{m}{2}\geq \frac{\eps^2}{2}n$, which is sufficient.

For $t>2$, the induction hypothesis implies that $G$ contains a good $(t-1)$-tuple $(v_1, \dots, v_{t-1})\in A^{t-1}$ with the common neighbourhood of size $|N(v_1, \dots, v_{t-1})|\geq 2\left(\frac{\eps}{6}\right)^{t-1} n$. Let $\gamma=(\frac{\eps}{6})^{t-1}$, then $\rs(G)\leq (\frac{\eps}{6})^{d-1}mn\leq \frac{\eps}{6}\gamma\cdot mn$ and $|X|\geq \gamma n$ are satisfied, so we can apply Lemma~\ref{lemma:one step embedding} to $X=N(v_1, \dots, v_{t-1})$. We get that there exists a vertex $v_{t}\in A$ for which $|N(v_1, \dots, v_{t})|\geq \frac{\eps}{6} |X|\geq (\frac{\eps}{6})^t n$ and $N(v_1, \dots, v_{t-1})\supsetneq N(v_1, \dots, v_{t})$. We conclude that $(v_1, \dots, v_{t})$ is a good $t$-tuple with sufficiently large common neighbourhood, completing the proof.
\end{proof}

Now, we are ready to prove Lemma~\ref{lemma:asymmetric upper bounds graphs eps}.

\begin{proof}[Proof of the Lemma~\ref{lemma:asymmetric upper bounds graphs eps}.]
Assume that the statement is false, and let $G$ be a counterexample with $d$ minimal, and among those $\min\{m,n\}$ is also minimal. That is, $G$ has at least $\eps mn$ edges, $G$ contains no good $d$-tuple and $\rs(G)<(\frac{\eps}{100})^{d-1}mn$. Clearly, we have $\min\{m,n\}\geq 2$. Also, if $d=2$, finding a good $d$-tuple corresponds to finding a pair of vertices of $A$ with at least two common neighbours, i.e. finding a cycle of length $4$ in $G$. Since $|E(G)|\geq \eps mn\geq 100\max\{n^{-1/2}, m^{-1}\}mn\geq mn^{1/2}+n$, a simple application of the classical K\H{o}v\'ari-S\'os-Tur\'an theorem shows that $G$ must have a cycle of length $4$, which suffices. Hence, we may assume that $d\geq 3$.

First, we show that $\min_{v\in A} \deg v\geq \left(1-\frac{2}{d}\right)\eps|B|$ and $\min_{v\in B} \deg v\geq \left(1-\frac{2}{d}\right)\eps|A|$. Since $A$ and $B$ are symmetric, we only discuss the case when $G$ has a vertex $v\in A$ with degree $\deg v\leq \left(1-\frac{2}{d}\right)\eps|B|$. Define the graph $G'=G-\{v\}$, which has parts of size $m'=m-1$ and $n'=n$. Furthermore, $G'$ has $\eps'm'n'$ edges, where $\eps'm'n'\geq \eps mn-\left(1-\frac{2}{d}\right)\eps n$. Thus, as $G'$ does not contain a good $d$-tuple and by the minimality of $G$, we conclude
\begin{align*}
\rs(G)&\geq \rs(G')\geq \left(\frac{\eps'}{100}\right)^{d-1}m'n'\geq \left(\frac{\eps mn-\left(1-\frac{2}{d}\right)\eps n}{100(m-1)n}\right)^{d-1}(m-1)n \\
&= \left(\Big(\frac{\eps}{100}\Big)^{d-1}mn\right)\cdot \frac{m-1}{m}\left(\frac{m-(1-\frac{2}{d})}{m-1}\right)^{d-1}\geq \Big(\frac{\eps}{100}\Big)^{d-1}mn.
\end{align*}
The last inequality is true due to simple calculations. We should also verify that $G'$ satisfies the conditions of Lemma~\ref{lemma:asymmetric upper bounds graphs eps} by showing that $\eps'^{d-1}m'\geq 100^{d-1}$ and $\eps'^d n'\geq 100^d$. But this also follows, since we have already shown that $\eps'^{d-1}m'n'\geq \eps^{d-1}mn$, which implies both $\eps'^{d-1}m'\geq\eps^{d-1} m\geq 100^{d-1}$ and $\eps'^{d}n'\geq \eps^{d}n\geq 100^d$ (since $\eps'^{d-1}m'n'\geq \eps^{d-1}mn$ also implies $\eps'\geq \eps$). We conclude that $\rs(G)\geq (\eps/100)^{d-1}mn$, a contradiction.

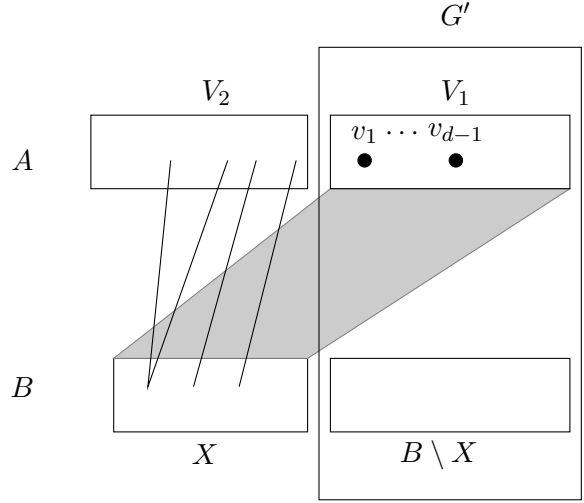
\begin{figure}
    \begin{center}
	\begin{tikzpicture}[scale=1.5]
            \node at (-2.8,1) {$A$};  \node at (-2.8,-1) {$B$}; 
            \draw (-2,-1.4) rectangle (-0.3,-0.75) ; \node at (-1.2,-1.6) {$X$};
            \draw (-0.1,-1.4) rectangle (2,-0.75) ;  \node at (0.85,-1.6) {$B\setminus X$};

             \node[vertex,label=above:$v_1$] (v1) at (0.2,1) {};
            \node[label=above:$\cdots$] (d) at (0.55,1) {};
            \node[vertex,label=above:$v_{d-1}$] (v2) at (1,1) {};

             \draw[draw=gray,fill=white!80!black]  (-0.1,0.75) --  (-2, -0.75) -- (-0.3,-0.75) -- (2,0.75) -- (-0.1,0.75) ;

            \draw (-1.5,1) -- (-1.7, -1) -- (-1, 1) ;
            \draw (-0.4,1) -- (-0.9,-1) ; \draw (-0.75,1) -- (-1.3,-1) ; 

            \draw (-2.2, 1.4) rectangle (-0.3, 0.75) ; \node at (-1.1,1.6) {$V_2$};
            \draw (-0.1, 1.4) rectangle (2,0.75) ; \node at (1,1.6) {$V_1$};
            \draw (-0.2, -2) rectangle (2.1, 2) ; \node at (1, 2.3) {$G'$};

        \end{tikzpicture}
    \end{center}
    \caption{An illustration for the proof of Lemma~\ref{lemma:asymmetric upper bounds graphs eps}}
    \label{fig2}
\end{figure}

In what follows, we assume that $G$ contains no vertex $v\in A$ of degree less than $\left(1-\frac{2}{d}\right)\eps n$ and no vertex $v\in B$ of degree less than $\left(1-\frac{2}{d}\right)\eps m$. We choose a good $(d-1)$-tuple $(v_1, \dots, v_{d-1})$ with the maximal size of the common neighbourhood, which we denote by $X=N(v_1, \dots, v_{d-1})$. By Lemma~\ref{lemma:good d-1 tuples}, there must be a good $(d-1)$-tuple with $|X|\geq \left(\frac{\eps}{6}\right)^{d-1} |B|$. If $A$ does not contain a good $d$-tuple in $G$, then any vertex $v\in A$ either has $X\subseteq N(v)$ or $|N(v)\cap X|=1$. Lemma~\ref{lemma:one step embedding} applied with $\gamma=(\frac{\eps}{6})^{d-2}$ shows that $|X|\leq \left(\frac{\eps}{6}\right)^{d-2} |B|$. We conclude that one can partition $A$ into two sets, $V_1$ and $V_2$, where $V_1=\{v\in A:X\subseteq N(v)\}$ and $V_2=\{v\in A:|X\cap N(v)|\leq 1\}$. 

Let us consider the edges incident to $X$. By our assumption on the minimum degree of vertices of $X$, we have $e(X, A)\geq |X|\cdot \left(1-\frac{2}{d}\right)\eps m\geq \frac{\eps}{3} m |X|$. On the other hand, \[\frac{\eps }{3}m|X|\leq e(X, A)=e(X, V_1)+e(X, V_2)\leq |X||V_1|+|V_2|\leq |X||V_1|+m.\] Since $|X|\geq \left(\frac{\eps}{6}\right)^{d-1} n$ and $(\frac{\eps}{100})^{d}n\geq 1$, we have $|X|\geq \frac{12}{\eps}$. Thus,  $\frac{\eps}{12} m|X|\geq m$ and so $\frac{\eps}{4} m |X|\leq |X||V_1|$, showing that $|V_1|\geq \frac{\eps}{4} m$.

Let us now consider the bipartite graph $G'=G[V_1\cup (B\backslash X)]$. The number of edges of this graph is at least
\begin{align*}
    e(G')&=e(V_1, B)-e(V_1, X)\geq |V_1|\cdot \left(1-\frac{2}{d}\right)\eps|B|-|V_1||X|\\
    &\geq |V_1|\left(\frac{d-2}{d}\eps|B|-\left(\frac{\eps}{6}\right)^{d-2}|B|\right)
    \geq \left(1-\frac{5/2}{d}\right)\eps|B||V_1|,
\end{align*}
since $6^{d-2}\geq 2d$ for all $d\geq 3$. Thus, the sizes of the vertex classes of $G'$ are $m'=|V_1|\geq \frac{\eps m}{4}$ and $n'=|B\backslash X|\geq n-(\frac{\eps}{6})^{d-2}n\geq \frac{5}{6}n$. Finally, the above computation shows that $e(G')=\eps'm'n'$, where $\eps'\geq (1-\frac{5/2}{d})\eps $.

Next, we argue that $G'$ does not have a good $(d-1)$-tuple. Suppose for contradiction that $G'$ contains a good $(d-1)$-tuple $(u_1, \dots, u_{d-1})$. The vertices $(u_1, \dots, u_{d-1})$ form a good $(d-1)$-tuple in the graph $G$ as well, and their common neighbourhood contains both $X$ and at least two vertices from $B\backslash X$ (since this is a good $(d-1)$-tuple in $G'$). But this implies $u_1, \dots, u_{d-1}$ is a $(d-1)$-tuple with a larger common neighbourhood in $G$ than $v_1, \dots v_{d-1}$, contradicting the maximality. Thus, we conclude that $G'$ cannot contain a good $(d-1)$-tuple.

Furthermore, we show that $\eps'^{d-2}m'\geq 100^{d-2}$ and $\eps'^{d-1} n'\geq 100^{d-1}$. This is a consequence of an easy computation since $$\eps'^{d-2}m'\geq \left(1-\frac{5/2}{d}\right)^{d-2}\eps^{d-2}\frac{\eps m}{4}\geq \frac{e^{-5/2}}{4} 100^{d-1}\geq 100^{d-2}.$$ 
In the above calculation, we used that $\eps^{d-1}n\geq 100^{d-1}$ and that $\big(1-\frac{5/2}{d}\big)^{d-2}\geq e^{-5/2}$, which holds since $\big(1-\frac{5/2}{d}\big)^{d-2}$ is a decreasing function for $d\geq 3$ with the limit $e^{-5/2}$ when $d\to \infty$. A similar computation shows that $\eps'^{d-1} n'\geq \big(1-\frac{5/2}{d}\big)^{d-1} \eps^{d-1}\cdot \frac{5}{6}n\geq 100^{d}\eps^{-1} \big(1-\frac{5/2}{3}\big)^{2}\cdot \frac{5}{6}\geq 100^{d-1}$, where we have used that $\big(1-\frac{5/2}{d}\big)^{d-1}$ is an increasing function of $d$, when $d\geq 3$. As $G$ is a counterexample with $d$ minimal, $G'$ is not a counterexample, so \[\rs(G')\geq \left(1-\frac{5/2}{d}\right)^{d-2}\left(\frac{\eps}{100}\right)^{d-2}|V_1||B\backslash X|\geq e^{-5/2}\frac{\eps^{d-1}}{4\cdot 100^{d-2}}m\cdot \frac{5}{6}n\geq \frac{\eps^{d-1}}{100^{d-1}}mn.\]
Hence, we have $\rs(G)\geq \rs(G')\geq (\frac{\eps}{100})^{d-1}mn$, which is a contradiction.
\end{proof}

\begin{proof}[Proof of the lower bounds in Theorem~\ref{thm:main thm rs_upper}.]
Proving the lower bounds on $\rs_d(m, n, I)$ and $\rs_d(n, m, I)$ is completely symmetric, so we just discuss $\rs_d(m, n, I)$. If $\eps>100\max\{m^{-\frac{1}{d-1}}, n^{-\frac{1}{d}}\}$, let us consider any sets of $m$ points and $n$ hyperplanes, say $\cP$ and $\cH$. The incidence graph $G=G(\cP, \cH)$ avoids the pattern $\Pi_d$ by Lemma~\ref{lemma:forbiddengraphs} and so $G$ contains no good $d$-tuples. Since $G$ has at least $\eps mn$ edges, Lemma~\ref{lemma:asymmetric upper bounds graphs eps} implies $\rs(G)\geq \left(\frac{\eps}{100}\right)^{d-1}mn$.
\end{proof}

\subsection{Constructions $K_{s, s}$-free incidence graphs}

The main goal of this section is to provide examples which show that the bounds given in Theorem~\ref{thm:main thm K_s,s} are tight. The main building block of all constructions in this section are \emph{subspace evasive sets}. Throughout this section, we consider only finite fields $\bF$.

\begin{definition}
A set of points $S\subset \bF^d$ is \emph{$(k, s)$-subspace evasive} if every $k$-dimensional affine subspace of $\bF^d$ contains less than $s$ elements of $S$.
\end{definition}

It is clear that the size of a  $(k,s)$-subspace evasive set is at most $O_s(|\mathbb{F}|^{d-k})$, as one can partition $\bF^d$ into $|\mathbb{F}|^{d-k}$ affine subspaces of dimension $k$. It was proved by Dvir and Lovett \cite{DL12} that for sufficiently large $s$, this trivial bound is optimal. Sudakov and Tomon \cite{ST23} gave an alternative proof using the random algebraic method.

\begin{lemma}[\cite{DL12}]\label{lemma:evasive}
For integers $1\leq k\leq d$, there exists $s\leq d^d$ such that for every large enough finite field $\mathbb{F}$, the space $\mathbb{F}^d$ contains a $(k, s)$-subspace evasive set of size at least $|\bF|^{d-k}$.
\end{lemma}

To be more precise, Dvir and Lovett \cite{DL12} construct a $(k, s)$-subspace evasive set of size at least $\frac{1}{3}|\bF|^{d-k}$. However, taking the union of four random translates of such a set results in a $(k, 4s)$-subspace evasive set of size at least $|\bF|^{d-k}$ with high probability. Removing the factor $1/3$ makes our calculations a bit nicer, but other than that has no real effect.

We prove the following equivalent formulation of the upper bounds in Theorem \ref{thm:main thm K_s,s}.

\begin{theorem}\label{thm:lower_1}
For every dimension $d$, there exist $s=s(d)$ and $c=c(d)>0$ such that the following holds. For every  $\alpha>0$ and sufficiently large integer $m$, there exists a field $\bF$, a set of $m$ points $\cP$ and a set of $n=\lfloor m^{\alpha}\rfloor$ hyperplanes $\cH$ in $\bF^d$ such that the incidence graph of $(\cP, \cH)$ is $K_{s, s}$-free and  
$$I(\cP, \cH)\geq \begin{cases} m  &\mbox{ if } \alpha\in (0,\beta_1],\\
cmn^{1-\frac{1}{t}} &\mbox{ if }\alpha\in [\alpha_t,\beta_t]\mbox{ for some }t\in \{2,\dots,d\},\\
cm^{1-\frac{1}{d+2-t}}n&\mbox{ if }\alpha\in [\beta_{t-1},\alpha_t]\mbox{ for some }t\in \{2,\dots,d\},\\
n&\mbox{ if } \alpha\in [\beta_d,\infty).\end{cases}$$
\end{theorem}

\noindent
The main building block of the proof of this theorem is the following lemma.

\begin{lemma}\label{lemma:construction1}
Let $2\leq t\leq d$ be integers, then there exists an integer $s=s(d)>0$ such that for every prime $p$ the following holds. There exists a set of points $\cP$ and a set of hyperplanes $\cH$ in $\bF_p^d$ such that $G(\cP, \cH)$ does not contain $K_{s, s}$, $|\cP|\geq p^t$, $|\cH|\geq p^{d-t+2}/s$, and \[I(\cP, \cH)=\frac{|\cP||\cH|}{p}\geq \Omega_d\left((|\cP||\cH|)^{1-\frac{1}{d+2}}\right).\]
\end{lemma}
\begin{proof}
 By Lemma \ref{lemma:evasive}, there exist a constant $s=s(d)$ and sets $\cP,\cN_0\subset \mathbb{F}_p^d$ such that $|\cP|\geq p^{t}$, $|\cN_0|\geq p^{d-t+1}$, $\cP$ is $(d-t, s)$-subspace evasive and $\cN_0$ is $(t-1, s)$-subspace evasive. In other words, any $s$ points of $\cP$ span an affine subspace of dimension at least $d-t+1$, and any $s$ points of $\cN_0$ span an affine subspace of dimension at least $t$. 

In particular, no line through the origin contains more than $s$ points of $\cN_0$. Thus, by keeping at most one point in each line through the origin, we get a subset $\cN$ of $\cN_0$ of size at least $\geq \frac{1}{s}|\cN_0|$.

Let $\cH$ be the set of all hyperplanes whose normal vectors are from $\cN$, i.e. all hyperplanes of the form $\langle w, x\rangle =b$ for $w\in \cN, b\in \bF_p$. Since no line through the origin contains more than one point of $\cN$, we observe that every pair $(w, b)\in \cN\times \bF_p$ yields a different hyperplane. Thus, the size of $\cH$ is $|\cH|=p|\cN|\geq p^{d-t+2}/s$. 

To count the incidences between $\cP$ and $\cH$, we observe that for every $w\in \cN$ and every $x\in \cP$, there is a unique $b\in \mathbb{F}_p$ such that $\langle w, x\rangle =b$. This means that out of $p$ hyperplanes of $\cH$ with normal vector $w\in \cN$, exactly one is incident to $x$. Hence, the number of incidences between $\cP$ and $\cH$ is 
$I(\cP, \cH)=|\cP||\cN|=\frac{1}{p}|\cP||\cH|$ and so
$I(\cP, \cH)\geq \Omega_d\left((|P||\cH|)^{1-\frac{1}{d+2}}\right).$

Finally, we argue that the incidence graph $G(\cP, \cH)$ is $K_{s,s}$-free. Indeed, any $s$ points of $\cP$ span an affine subspace of dimension at least $d-t+1$. On the other hand, the intersection of $s$ hyperplanes is nonempty only when no two of them are parallel, meaning that their normal vectors are distinct. If these $s$ hyperplanes intersect in a $(d-t+1)$-dimensional affine space, this means that the $s$ normal vectors lie in a $(t-1)$-dimensional linear space orthogonal to it. But this is not possible since $\cN$ is $(t-1, s)$-subspace evasive. Hence, the intersection of $s$ hyperplanes can be at most $(d-t)$-dimensional. As $d-t<d-t+1$, this shows that $G(\cP,\cH)$ is $K_{s,s}$-free.
\end{proof}

Observe that Lemma~\ref{lemma:construction1} implies the existence of $\cP, \cH$  with $|\cP|=m_0=\Omega_d(p^t), |\cH|=n_0=\Omega_d(p^{d+2-t})$ and $I(\cP, \cH)=\Omega_d((m_0n_0)^{1-\frac{1}{d+2}})$. This only covers a sparse set of potential values of $m$ and $n$. However, by using an approximation and a subsampling trick, one can cover all possible values of $m$ and $n$.

\begin{proof}[Proof of Theorem~\ref{thm:lower_1}]
Let us first discuss the boundary cases, i.e. when $\alpha\in (0,\beta_1]$ or $\alpha\in [\beta_d,\infty)$. In this case, we show that there exist sets $\cP$ and $\cH$ with $|\cP|=m, |\cH|=n=m^\alpha$, $I(\cP, \cH)\geq \max\{m, n\}$ and such that $G(\cP, \cH)$ does not contain $K_{2, 2}$. In the case $\alpha\in [0,\beta_1]$, it suffices to choose $p\geq \max\{m, n\}$ and take $\cH$ to be $n$ parallel hyperplanes in $\bF_p^d$ and $\cP$ to be $m$ points on one of the hyperplanes. On the other hand, when $\alpha\in [\beta_d,\infty)$, let $\cP$ be $m$ collinear points and $\cH$ be a set of $n$ hyperplanes containing exactly one of these points.

Let us now present the main argument. The idea of the proof is to pick a configuration $(\cP_0, \cH_0)$ given by Lemma~\ref{lemma:construction1} and randomly sample a proportion of it to obtain the required values of $m, n$. 

First, consider the case when $\alpha_t \leq \alpha\leq \beta_t$ for some $t\in \{2, \dots, d\}$. Let $p$ be the smallest prime larger than $2(sn)^{1/t}$. By Bertrand's postulate, we have $p\leq 4(sn)^{1/t}$. By applying Lemma~\ref{lemma:construction1} with $d+2-t$ instead of $t$, we obtain a set of points $\cP_0$ and a set of hyperplanes $\cH_0$ in $\bF_p^d$ with $|\cP_0|=m_0\geq p^{d+2-t}, |\cH_0|=n_0\geq p^{t}/s$ and $I(\cP_0, \cH_0)=\frac{1}{p}|\cP_0||\cH_0|$. Moreover, $G(\cP_0, \cH_0)$ does not contain $K_{s, s}$ as a subgraph.

Let $\cP$ and $\cH$ be random subsets of $\cP_0, \cH_0$ containing $m$ and $n$ elements, respectively. Observe that $n_0\geq \frac{1}{s}p^t\geq n$ and $m_0\geq p^{d+2-t}\geq m^{\frac{d+2-t}{t}}\geq n^{1/\alpha}=m$, so it indeed makes sense to consider such subsets of $\cP_0, \cH_0$. The expected number of incidences between $\cP$ and $\cH$ is simply $\bE[I(\cP, \cH)]=\frac{m}{m_0}\cdot \frac{n}{n_0} I(\cP_0, \cH_0)\geq \frac{1}{p}mn$. Since $p\leq 4(sn)^{-1/t}$, picking a pair $(\cP, \cH)$ with at least as many incidences as the expectation, we get
\[I(\cP, \cH)\geq \frac{1}{p}mn\geq \frac{1}{4s^{1/t}} n^{-\frac{1}{t}}\cdot mn.\]

In the case $\beta_{t-1}\leq \alpha\leq \alpha_t$, we proceed similarly. We pick $p$ to be the smallest prime larger than $2s^{\frac{1}{t}}m^{\frac{1}{d+2-t}}$, then $p<4s^{\frac{1}{t}}m^{\frac{1}{d+2-t}}$. We apply Lemma~\ref{lemma:construction1} with $d+2-t$ instead of $t$ to find a set of points $\cP_0$ and a set of hyperplanes $\cH_0$ in $\bF_p^d$ with $|\cP_0|=m_0\geq p^{d+2-t}, |\cH_0|=n_0\geq p^{t}/s$ and $I(\cP_0, \cH_0)=\frac{1}{p}|\cP_0||\cH_0|$. By an almost identical computation as before one can check that $m_0\geq m$ and $n_0\geq m^{\alpha}\geq n$. Hence, if $\cP\subset \cP_0$ is a random $m$-element subset and $\cH\subset \cH_0$ is a random $n$-element subset, they satisfy $\mathbb{E}(I(\cP, \cH))\geq \frac{1}{p}mn\geq \frac{1}{4s^{1/t}} m^{-\frac{1}{d+2-t}}\cdot mn$. Therefore, we can pick a pair $(\cP, \cH)$ with at least as many incidences as the expectation, completing the proof.
\end{proof}

\subsection{Constructions of dense incidence graphs with no large bipartite subgraphs}

In this section, we prove the upper bounds of Theorem~\ref{thm:main thm rs}. However, instead of assuming $m\leq n$ and giving bounds on both $\rs_d(m, n, I)$ and $\rs_d(n, m, I)$, as originally stated in Theorem~\ref{thm:main thm rs}, we prefer to keep the notation consistent by always denoting the number of hyperplanes by $m$ and the number of points by $n$. Thus, we will give bounds on $\rs_d(m, n, I)$ which depend on $\max\{m, n\}$ and $\min\{m, n\}$. With this notational change, Theorem~\ref{thm:main thm rs} can be equivalently restated as follows.

\begin{theorem}\label{thm:lower_2}
Let $d, m, n$ be positive integers, $\eps\in (0,1)$ and $I=\eps mn$. 
\begin{itemize}
    \item If $\eps\geq \frac{1}{4}\max\{m, n\}^{-1/d}$ and $\eps\geq \frac{1}{4}\min\{m, n\}^{-1/(d-1)}$, then $\rs_d(m, n, I)\leq O_d(\eps^{d-1}mn).$
    \item If $\eps<\frac{1}{4}\max\{m, n\}^{-1/d}$ or $\eps<\frac{1}{4}\min\{m, n\}^{-1/(d-1)}$ then  $\rs_d(m, n, I)\leq O_d(\eps \max\{m, n\}).$
\end{itemize}
\end{theorem}
The implied constant depending on the dimension in the above theorem is at most $2^{O(d\log d)}$. The main geometric construction is contained in the following lemma, which can be thought of as a strengthening of Lemma~\ref{lemma:construction1}. 

\begin{lemma}\label{lemma:construction2}
Let $p$ be a prime and let $d, t, k, \ell$ be positive integers satisfying $t\leq d-1$ and $k, \ell\leq p^{d-t}$. Then, there exist an integer $s=O(d^d)$, a set of points $\cP$ and a set of hyperplanes $\cH$ in $\bF_p^d$ satisfying the following properties:
\begin{itemize}
    \item[(i)] $kp^t/2\leq |\cP|\leq kp^{t}$, $\frac{1}{4}\ell p^t\leq |\cH|\leq \ell s p^t$, and every point of $\cP$ is contained in exactly $|\cH|/p$ hyperplanes of $\cH$. In particular, $I(\cP, \cH)=\frac{|\cP|\cdot |\cH|}{p}$.
    \item[(ii)] For  $a\in \{1,\dots, t-1\}$ and any set $A\subseteq \cP$ of size $|A|\geq skp^a$, $A$ is not contained in a $(d-t+a)$-dimensional affine subspace. 
    \item[(iii)] For $b\in \{1,\dots,t-2\}$ and any set $B\subseteq \cH$ of size $|B|\geq s^2\ell p^b$, the intersection $\bigcap_{H_i\in B} H_i$ does not contain a $(t-1-b)$-dimensional affine subspace. 
\end{itemize}
\end{lemma}
\begin{proof}
By Lemma \ref{lemma:evasive}, there exists $s=O(d^d)$ such that for every prime $p$ there exists a $(d-t, s)$-subspace evasive set $\cP_0$ containing $p^t$ points of $\bF_p^d$. In other words, any $s$ points of $\cP_0$ span an affine subspace of dimension at least $d-t+1$. Furthermore, there exists a set of points $\cQ_0$ containing  $p^{t-1}$ points of $\bF_p^d$ such that $\cQ_0$ is $(d-t+1, s)$-subspace evasive, which means that any $s$ points of $\cQ_0$ span an at least $(d-t+2)$-dimensional affine subspace.

The set $\cP$ is constructed as the union of $k$ random translates of $\cP_0$. More precisely, we choose $k$ random vectors $u_1, \dots, u_k\in \bF_p^d$ (with repetition) and set $\cP=\bigcup_{i=1}^k (\cP_0+u_i)$. The set $\cP$ constructed in this way is $(d-t, ks)$-subspace evasive, since any affine space of dimension $d-t$ containing $ks$ elements of $\cP$ must contain at least $s$ elements of one of the translates $\cP_0+u_i$. But this is impossible, since $\cP_0$ is a $(d-t, s)$-subspace evasive set.

Next, we compute the expected size of $\cP$ and argue that $u_1, \dots, u_k$ can be chosen such that $|\cP|\geq \frac{1}{2}kp^t$. For a given point $x\in \bF_p^d$, the probability that $x$ lies in $\cP_0+u_i$ is $$\Pb[x\in (\cP_0+u_i)]=\Pb[u_i\in (x-\cP_0)]=\frac{1}{p^d}|\cP_0|.$$ 
Thus, the expected size of $\cP$ is
\begin{align*}
\bE\big[|\cP|\big]=\sum_{x\in \bF_p^d} \Pb[x\in \cP]&=\sum_{x\in \bF_p^d} \Big(1-\prod_{i=1}^k \Pb[x\notin (\cP_0+u_i)]\Big)=\sum_{x\in \bF_p^d} \left(1-\Big(1-\frac{|\cP_0|}{p^d}\Big)^k\right)\\
&\geq p^d\Big(1-\exp\Big(-\frac{k|\cP_0|}{p^d}\Big)\Big)\geq p^d\frac{k|\cP_0|}{2p^d}=\frac{1}{2}kp^t.
\end{align*}
In the above calculation, we have used the inequalities $1-\frac{x}{2}\geq e^{-x}\geq 1-x$ for $x\in (0, 1)$. Therefore, we conclude that $u_1, \dots, u_k$ can be chosen such that $\cP$ has size at least $\frac{1}{2}kp^t$. On the other hand, we always have $|\cP|\leq k|\cP_0|=kp^t$.

Let us now explain how to process the set $\cQ_0$ and how to construct from it the set of hyperplanes $\cH$. Let $U=\{x\in \bF_p^d:x(d)=0\}$. Since we may translate the set $\cQ_0$ without altering its properties, we may assume that at most $\frac{|\cQ_0|}{p}$ elements of $\cQ_0$ lie in $U$. Then, we define the set $\cQ_1=\cQ_0\backslash U$, which is still $(d-t+1, s)$-subspace evasive and has size $|\cQ_1|\geq |\cQ_0|/2$.

Pick $s\ell$ uniformly random vectors $v_1, \dots, v_{s\ell}\in U$ and define $\cQ_2=\bigcup_{i=1}^{s\ell} (\cQ_1+v_i)$. In a similar manner as above, we observe that $\cQ_2$ is a $(d-t+1, \ell s^2)$-subspace evasive set. Our goal now is to show that $v_1, \dots, v_{s\ell}$ can be chosen in a way such that at least $\frac{1}{4}\ell p^{t-1}$ lines through the origin contain a point of $\cQ_2$. If $L\subset \bF_p^d$ is a fixed line through the origin, note that $L$ intersects $\cQ_1+v_i$ if and only if the translate of $L$ given by $L-v_i$ intersects $\cQ_1$. Note that union of all the translates of $L$ covers $\cQ_1$.
Since $\cQ_1$ is a $(d-t+1, s)$-subspace evasive set, each translate of $L$ contains at most $s$ points of $\cQ_1$ and so at least $|\cQ_1|/s$ translates of $L$  intersect $\cQ_1$. Thus, we conclude $\Pb[L\cap (\cQ_1+v_i)\neq \emptyset]\geq \frac{|\cQ_1|/s}{p^{d-1}}$. If we denote the number of lines through the origin intersecting $\cQ_2$ by $Y$, by a similar calculation as above we have 
\[\bE[Y]=\sum_{L\text{ line through }0} 1-\Pb[L\cap \cQ_2=\emptyset]\geq
\sum_{L\text{ line through }0} 1-\left(1-\frac{|\cQ_1|}{sp^{d-1}}\right)^{s\ell}\geq p^{d-1}\frac{\ell|\cQ_1|}{2p^{d-1}}\geq \frac{1}{4}\ell p^{t-1}.\]

We form the subset $\cQ\subseteq \cQ_2$ by retaining at most one point on each line through the origin. The above calculation implies that $v_1, \dots, v_{s\ell}$ can be chosen such that $|\cQ|\geq \frac{1}{4}\ell p^{t-1}$. On the other hand, the size of $\cQ$ is bounded by $|\cQ|\leq |\cQ_2|\leq s\ell |\cQ_0|\leq s\ell p^{t-1}$ and $\cQ$ is a $(d-t+1, \ell s^2)$-subspace evasive set.

Finally, we choose $\cH$ to be the set of hyperplanes given by the equations $\langle q, x\rangle =b$, for $q\in \cQ$ and $b\in \bF_p$. We ensured that no two points $q_1, q_2\in \cQ$ lie on the same line through the origin, meaning that the hyperplanes given by $\langle q_1, x\rangle =b_1$, $\langle q_2, x\rangle =b_2$ are distinct for all $q_1, q_2\in \cQ$ and $b_1, b_2\in \bF_p$.

All that is left is to show that the conditions of the theorem hold for the sets $\cP, \cH$. We have already shown that $|\cP|\in [kp^t/2,kp^{t}]$. Furthermore, since every pair $(q, b)\in \cQ\times \bF_p$ yields a different hyperplane of $\cH$, we conclude $|\cH|=p|\cQ|\in \left[\frac{1}{4}\ell p^{t}, s\ell p^{t}\right]$. Finally, for every $x\in \cP$ and every $q\in \cQ$, there is a unique $b$ for which $\langle x, q\rangle =b$, and therefore the number of hyperplanes containing $x$ is precisely $|\cQ|=|\cH|/p$. This verifies (i).

Now, we prove that for any $a\in \{1,\dots,t-1\}$, no $(d-t+a)$-dimensional affine subspace contains $skp^a$ points of $A$. Suppose this is not the case and there was a set $A\subseteq P$ of size $|A|\geq skp^a$ which was contained in a $(d-t+a)$-dimensional affine space $V$. Note that $V$ can be partitioned into $p^a$ translates of some $(d-t)$-dimensional affine subspace $V'<V$. But then there exists a translate of $V'$ containing at least $ks$ points of $\cP$.This is a contradiction since $\cP$ is a $(d-t, ks)$-subspace evasive set, thus establishing (ii).

We proceed by a similar argument to prove (iii). Let $B$ be a set of $s^2\ell p^b$ hyperplanes, denoted by $H_1, \dots, H_{s^2\ell p^b}$. Furthermore, we assume that the equation of $H_i$ is given by $\langle q_i, x\rangle =b_i$. Observe that if the intersection of the hyperplanes $H_i$ is nonempty, one must have $q_i\neq q_j$ for all $i, j\in [s^2\ell p^b]$, since otherwise two parallel hyperplanes belong to $B$. Let us denote the set of all points $q_i$ for $i\in [s^2\ell p^b]$ by $B'\subseteq \cQ$. Assume, for contradiction, that $\bigcap_{i=1}^{s^2\ell p^b}H_i$ contains a $(t-1-b)$-dimensional affine subspace. This affine subspace can be written as $V+q$ for some $q\in \bF_p^d$ and some $(t-1-b)$-dimensional linear subspace $V$. Then all points $q_i$ are contained in the orthogonal complement of $V$, which we denote by $V^{\perp}$. But $V^{\perp}$ is a $(d-t+1+b)$-dimensional linear space, which contains $s^2\ell p^b$ points of $\cQ$. By partitioning this space into $p^b$ translates of a $(d-t+1)$-dimensional affine space, and using that $\cQ$ is a $(d-t+1, \ell s^2)$-subspace evasive set, one arrives at a contradiction as above.
\end{proof}

In what follows, we prove Theorem \ref{thm:lower_2} in a series of lemmas, addressing different regimes.

\begin{lemma}\label{lemma:construction3}
Let $\eps\in (0,1)$, and let $d$ and $m, n$ be integers which satisfy 
$$\eps\leq \frac{1}{4} \max\{m, n\}^{-\frac{1}{d}}.$$
Then there exists a prime $p$, a set $\cP$ of $m$ points and a set $\cH$ of $n$ hyperplanes in $\bF_p^d$ such that $I(\cP, \cH)\geq \eps mn$ and $\rs(\cP, \cH)\leq (2s)^4\eps \max\{m, n\}$, where $s$ is given by Lemma \ref{lemma:construction2}.
\end{lemma}
\begin{proof}
For notational convenience, we assume that $\max\{m, n\}=n$, noting that the same argument applies if $\max\{m, n\}=m$. Let $p$ be a prime satisfying $\eps^{-1}/2<p\leq \eps^{-1}$, and set $t=\lfloor \log_p 4n\rfloor$ and  $\ell=\lceil \frac{4n}{p^t}\rceil\leq p$. Note that $4n\leq \ell p^t\leq 8n$. Since $p> \eps^{-1}/2\geq 2n^{1/d}$, we have $p^d> 4n$ and so $t\leq d-1$. Thus, Lemma~\ref{lemma:construction2} can be applied with $k=\ell$ to get a set of points $\cP_0$ and a set of hyperplanes $\cH_0$ with $m_0=|\cP_0|\geq \ell p^t/2$, $n_0=|\cH_0|\geq \ell p^t/4$ and $I(\cP_0, \cH_0)=\frac{1}{p} m_0n_0$, satisfying (i), (ii) and (iii).

Let us now argue that $\rs(\cP_0, \cH_0)\leq (2s)^4 \eps n$. Let $A\subseteq \cP_0, B\subseteq \cH_0$ such that every point of $A$ lies on every hyperplane in $B$. We consider three cases: either $|A|\geq s^2\ell p^{t-2}$, $|B|\geq s^2\ell p^{t-2}$ or $\max\{|A|, |B|\}< s^2\ell p^{t-2}$. In all cases, we show that \[|A|\cdot |B|\leq s^4 \ell p^{t-1} \leq 2 s^4 \eps \ell p^t \leq 16 s^4 \eps n.\]

If $|A|\geq s^2\ell p^{t-2}$, then by (ii) of Lemma~\ref{lemma:construction2}, the points of $A$ span an affine subspace of dimension at least $d-t+t-1=d-1$, which means that at most one of the hyperplanes contains them all. On the other hand, again by (ii) of Lemma~\ref{lemma:construction2}, any hyperplane can contain at most $s\ell p^{t-1}$ points of $\cP_0$ and thus $|A|\leq s\ell p^{t-1}$. In conclusion, we have $|A|\leq s\ell p^{t-1}$ and $|B|=1$, implying $|A|\cdot |B|\leq s\ell p^{t-1}$ as claimed.

The case $|B|\geq s^2\ell p^{t-2}$ is very similar. By (iii) of Lemma~\ref{lemma:construction2}, the intersection of hyperplanes of $B$ does not contain a line, meaning that $\bigcap_{H\in B} H$ is a single point. Furthermore, there are exactly $|\cH_0|/p$ hyperplanes of $\cH_0$ containing any point of $\cP_0$. Since $|\cH_0| \leq \ell s p^t$, we have $|B|\leq |\cH_0|/p\leq s\ell p^{t-1}$, and therefore $|A|\cdot |B|\leq s\ell p^{t-1}$.

Finally, if $\max\{|A|, |B|\}< s^2 \ell p^{t-2}$, we let $a$ and $b$ be the smallest integers for which $s^2\ell p^a\leq |A|<s^2\ell p^{a+1}$ and $s^2 \ell p^b\leq |B|<s^2\ell p^{b+1}$. By property (ii) of Lemma~\ref{lemma:construction2}, the points of $A$ span an at least $(d-t+a+1)$-dimensional affine space. Similarly, by property (iii) of Lemma~\ref{lemma:construction2}, the hyperplanes of $B$ intersect in at most a $(t-2-b)$-dimensional space. Since all points of $A$ belong to all hyperplanes of $B$, it must be that $d-t+a+1\leq t-2-b$, i.e. $a+b+2\leq 2t-d-1$. Thus, \[|A|\cdot |B|\leq s^4\ell^2 p^{a+b+2}\leq s^4\ell p\cdot p^{2t-d-1}\leq s^4\ell p^{t-1},\] 
where the last inequality holds by noting that $t\leq d-1$. This completes the third case.

To finish the proof, let $\cP$ be a random subset of $\cP_0$ containing $m$ points and $\cH$ be a random subset of $\cH_0$ containing $n$ hyperplanes. One should verify that $m\leq |\cP_0|$ and $n\leq |\cH_0|$, but this is clear since $|\cP_0|, |\cH_0|\geq \frac{1}{4}\ell p^t\geq n$ and $n\geq m$. We also have $\rs(\cP, \cH)\leq \rs(\cP_0, \cH_0)\leq (2s)^4\eps n$. Furthermore, the expected number of incidences between $\cP$ and $\cH$ is \[\bE\big[I(\cP, \cH)\big]=\frac{n}{n_0}\frac{m}{m_0}I(\cP_0, \cH_0)=\frac{1}{p}mn\geq \eps mn.\] Thus, there exist subsets $\cP, \cH$ of $\cP_0, \cH_0$ with at least $\eps mn$ incidences and $\rs(P, \cH)\leq (2s)^4\eps n$.
\end{proof}

\begin{lemma}\label{lemma:construction4}
Let $\eps\in (0,1)$, and let $d$ and $m\leq n$  be positive integers which satisfy  $$\eps\leq \frac{1}{4}m^{-\frac{1}{d-1}}.$$
Then there exists a prime power $q$, a set $\cP$ of $m$ points and a set $\cH$ of $n$ hyperplanes in $\bF_q^d$ such that $I(\cP, \cH)\geq \eps mn$ and $\rs(\cP, \cH)\leq (4s)^4\eps n$, where $s$ is given by Lemma \ref{lemma:construction2}.
\end{lemma}
\begin{proof}
The main idea of the proof is to take a symmetric $(d-1)$-dimensional construction described in Lemma~\ref{lemma:construction3} and ``blow it up" to obtain our desired configuration.

More precisely, we apply Lemma~\ref{lemma:construction3} to find a set of $m$ points $\cP_0$ and a set of $m$ hyperplanes $\cH_0$ in $\bF_p^{d-1}$ for some prime $p$ for which $I(\cP_0,\cH_0)\geq \eps m^2$ and $\rs(\cP_0, \cH_0)\leq (2s)^4\eps m$. The hyperplanes and points may be considered in the larger ambient space $\bF_q^{d-1}$, for any prime power $q=p^k$. Passing to $\bF_q^{d-1}$ does not change the incidence graph. In the rest of the proof, we fix $q$ to be a power of $p$ larger than $n/m$.

Now, for each $x\in \bF_q^{d-1}$, let $x'\in \bF_q^d$ be the point whose first $d-1$ coordinates are the same as $x$, and the last coordinate is $0$. Set $\cP=\{x'|x\in \cP_0\}$. Furthermore, we fix a set $S=\{s_1, \dots, s_\ell\}\subseteq \bF_q$ of size $\ell=\lceil n/m\rceil$. Then, for every hyperplane $H\in \cH_0$, if $H$ is given by the equation $\langle a, x\rangle =a_1x_1+\cdots +a_{d-1}x_{d-1} =b$, we define hyperplanes $H^{(1)}, \dots, H^{(\ell)}$ in $\bF_q^{d}$, where $H^{(i)}$ is given by the equations $a_1x_1+\cdots+a_{d-1}x_{d-1}+s_ix_d=b$. Then, we let $\cH_1=\{H^{(i)}|i\in [\ell], H\in \cH_0\}$.

The sizes of the newly constructed sets are $|\cP|=m, |\cH_1|=m\lceil \frac{n}{m}\rceil>n$. Furthermore, if $x\in \cP_0$ is incident to $H\in \cH_0$, then $x'$ is incident to all of the hyperplanes $H^{(1)}, \dots, H^{(\ell)}$. Thus, the number of incidences between $\cP$ and $\cH_1$ is $I(\cP, \cH_1)\geq \ell I(\cP_0, \cH_0)\geq \eps \ell |\cP_0||\cH_0|=\eps |\cP||\cH_1|$. Finally, the size of the largest complete bipartite graph is bounded by $$\rs(\cP, \cH_1)\leq \ell \cdot \rs(\cP_0, \cH_0)\leq \ell\cdot (2s)^4\eps m\leq (4s)^4\eps n.$$ Taking a random $m$ element subset $\cH$ of $\cH_1$, the expectation of $I(\cP,\cH)$ is at least $\eps mn$, so there is a choice for $\cH$ such that $(\cP,\cH)$ satisfies our desired conditions.
\end{proof}

\begin{lemma}\label{lemma:construction4'}
Let $\eps\in (0,1)$, and let $d$ and $m\geq n$  be positive integers which satisfy $$\eps\leq\frac{1}{4}n^{-\frac{1}{d-1}}.$$ Then there exists a prime power $q$, a set $\cP$ of $m$ points and a set $\cH$ of $n$ hyperplanes in $\bF_q^d$ such that $I(\cP, \cH)\geq \eps mn$ and $\rs(\cP, \cH)\leq (4s)^4\eps m$, where $s$ is given by Lemma \ref{lemma:construction2}.
\end{lemma}
\begin{proof}
The proof is verbatim the same as the proof of Lemma~\ref{lemma:construction4}, the only difference being that we ``blow up" points instead of hyperplanes. More precisely, we start from a set of $n$ points $\cP_0$ and a set of $n$ hyperplanes $\cH_0$ in $\bF_p^{d-1}$ for which $I(\cP_0,\cH_0)\geq \eps n^2$ and $\rs(\cP_0, \cH_0)\leq (2s)^4\eps n$ and consider them in $\bF_q^{d-1}$, where $q>m/n$.

Then, we fix a set $S=\{s_1, \dots, s_\ell\}\subseteq \bF_q$ of size $\ell=\lceil m/n\rceil$ and for each point $x\in \bF_q^{d-1}$ we define $x^{(i)}\in \bF_q^d$ be the point whose first $d-1$ coordinates are the same as $x$, and the last coordinate is $s_i$. Set $\cP_1=\{x^{(i)}|x\in \cP_0, i\in [\ell]\}$. Then, for every hyperplane $H\in \cH_0$, if $H$ is given by the equation $\langle a, x\rangle =a_1x_1+\cdots +a_{d-1}x_{d-1} =b$, we define the hyperplane $H'$ in $\bF_q^{d}$, given by the equation $a_1x_1+\cdots+a_{d-1}x_{d-1}+0\cdot x_d=b$. Then, we let $\cH=\{H'| H\in \cH_0\}$.

The sizes of the newly constructed sets are $|\cH|=n, |\cP_1|=n\lceil \frac{m}{n}\rceil>n$. Thus, by choosing a random subset $\cP\subseteq \cP_1$ and performing the same calculation as in the proof of Lemma~\ref{lemma:construction4} one arrives at the desired conclusion.
\end{proof}

\begin{lemma}\label{lemma:construction5}
Let $\eps\in (0,1)$, and let $d$ and $m, n$ be positive integers which satisfy $$4\eps\geq n^{-\frac{1}{d-2}}\text{ and } 4\eps\geq m^{-\frac{1}{d-2}}.$$ Then there exists a prime power $q$, a set $\cP$ of $m$ points and a set $\cH$ of $n$ hyperplanes in $\bF_q^d$ such that $I(\cP, \cH)\geq \eps mn$ and $\rs(\cP, \cH)\leq 2^6s^4(4\eps)^{d-1} mn$, where $s$ is given by Lemma \ref{lemma:construction2}.
\end{lemma}
\begin{proof}
The proof of this lemma is somewhat similar to the proofs of Lemmas~\ref{lemma:construction4} and \ref{lemma:construction4'}, since we  use the idea of ``blowing-up" a lower-dimensional construction. However, in this case we use a $(d-2)$-dimensional configuration as a starting point and we ``blow-up" both points and hyperplanes.

Let $p$ be a prime between $\eps^{-1}/2$ and $\eps^{-1}$, and set $N=(4\eps)^{-(d-2)}$. We remark that the inequalities $N\leq m, n$ are satisfied. Note that the condition of Lemma~\ref{lemma:construction3} is also satisfied with $d-2$ instead of $d$, and $N$ instead of $m$ and $n$. Therefore, we can find sets $\cP_0, \cH_0$ of points and hyperplanes in $\bF_p^{d-2}$, both of size $N$, and satisfying $I(\cP_0, \cH_0)\geq \eps N^2$ and $\rs(\cP_0, \cH_0)\leq (2s)^4\eps N$.

As in the proof of Lemma~\ref{lemma:construction4}, we consider $\cP_0, \cH_0$ in $\bF_q^{d-2}$ instead of $\bF_p^{d-2}$, where $q$ is the power of $p$ satisfying $q>\max\{\frac{m}{N}, \frac{n}{N}\}$. Furthermore, we fix two sets $S, T\subseteq \bF_q$, such that $S=\{s_1, \dots, s_k\}$ and $T=\{t_1, \dots, t_\ell\}$, where $k=\lceil \frac{m}{N}\rceil, \ell=\lceil \frac{n}{N}\rceil$. To every point $x\in \cP_0$, we associate $k$ points of $\bF_q^d$ as follows. For $i\in [k]$, let $x^{(i)}\in \mathbb{F}_q^{d}$  be equal to $x$ on the first $d-2$ coordinates, equal to $s_i$ on the coordinate $d-1$, and equal to $0$ on the last coordinate. Similarly, for a hyperplane $H\in \cH_0$ given by the equation $a_1 x_1+\cdots +a_{d-2}x_{d-2}=b$, we associate the hyperplanes $H^{(j)}$ in $\bF_q^{d}$ for $j\in [\ell]$ defined by the equations $a_1 x_1+\cdots +a_{d-2}x_{d-2}+0\cdot x_{d-1}+t_j x_d=b$. Finally, we set $\cP_1=\{x^{(i)}|i\in [k], x\in \cP_0\}$ and $\cH_1=\{H^{(j)}|H\in \cH_0, j\in [\ell]\}$.

If $x\in H$, then it is easy to check that $x^{(i)}\in H^{(j)}$ for all $i\in [k], j\in [\ell]$, and no other type of incidences emerge in $(\cP_1,\cH_1)$. Thus, $I(\cP_1, \cH_1)\geq k\ell I(\cP_0, \cH_0)\geq \eps k|\cP_0|\ell |\cH_0|=\eps |\cP_1||\cH_1|$. Furthermore, $$\rs(\cP_1, \cH_1)\leq k\ell \rs(\cP_0, \cH_0)\leq \frac{4mn}{N^2}\cdot (2s)^4\eps N= 2^6s^4\eps \frac{mn}{N}.$$ Recalling that $N=(4\eps)^{-(d-2)}$, we obtain $\rs(\cP_1, \cH_1)\leq 2^6s^4(4\eps)^{d-1}mn$. Subsampling a random $m$ element subset $\cP$ of $\cP_1$, and a random $n$ element subset $\cH$ of $\cH_1$, completes the proof.
\end{proof}

\begin{proof}[Proof of Theorem~\ref{thm:lower_2}.]
We claim that it suffices to take $C=4^{d+4}s^4=O((4d)^{4d})$. Consider three cases.

\begin{description}
    \item[Case 1.] $\eps\geq \frac{1}{4}\max\{m, n\}^{-1/d}$ and $\eps\geq \frac{1}{4}\min\{m, n\}^{-1/(d-1)}$.
 
    Then  $4\eps >m^{-1/(d-2)}$ and  $4\eps>n^{-1/(d-2)}$ are also satisfied, so Lemma~\ref{lemma:construction5} applies. Therefore, we have $\rs_d(m, n, \eps mn)\leq C\eps^{d-1}mn$.

    \item[Case 2.] $\eps\leq \frac{1}{4}\max\{m, n\}^{-1/d}$.

    In this case, Lemma~\ref{lemma:construction3}  shows that $\rs_d(m, n, \eps mn)\leq (2s)^4 \eps \max\{m, n\}\leq C\eps \max\{m, n\}$.

    \item[Case 3.] $\eps\leq \frac{1}{4}\min\{m, n\}^{-1/(d-1)}$.
    
    We apply either Lemma~\ref{lemma:construction4} or Lemma~\ref{lemma:construction4'}, depending on whether $m\geq n$ or $n\geq m$. This shows that $\rs_d(m, n, \eps mn)\leq (4s)^4 \eps \max\{m, n\}\leq C\eps \max\{m, n\}$.
\end{description}
\end{proof}

\section{Point-variety incidences}\label{sec:point-variety incidences}

\subsection{Basics of algebraic geometry}\label{sec:alg geo background}

Much of the following material is based on the Appendix of \cite{DGW} and the excellent book of Cox, Little and O'Shea \cite{CLO}.

Let $\oF$ be the algebraic closure of the field $\bF$. We denote by $\oF[x_1, \dots, x_D]$ the set of polynomials in variables $x_1, \dots, x_D$ with coefficients in $\oF$. An \textit{affine algebraic set} $V$ is a set of common zeros of polynomials $f_1, \dots, f_k\in \oF[x_1, \dots, x_D]$ in $\oF^D$, and this set is denoted by $V(f_1, \dots, f_k)$. Formally, \[V(f_1, \dots, f_k)=\left\{(a_1, \dots, a_d)\in \oF^D: f_i(a_1, \dots, a_d)=0\text{ for all }i\in [k]\right\}.\]

In a similar way, one can define projective algebraic sets. In this case, the ambient projective space $\bP^D(\oF)$ over the field $\oF$ is defined as the set of equivalence classes of points in $\oF^{D+1}\backslash\{(0, \dots, 0)\}$, where $(a_0, \dots, a_{D})\sim (b_0, \dots, b_D)$ if $a_i=\lambda b_i$ for every $i$ for some $\lambda\in \oF$. Note that for every homogeneous polynomial $f\in \oF[x_0, \dots, x_D]$ and every $\lambda\neq 0$, we have $f(a_0, \dots, a_D)=0$ if and only if $f(\lambda a_0, \dots, \lambda a_D)=0$. Hence, the set of zeros of a homogeneous polynomial in $\bP^D(\oF)$ is well-defined. Thus, we can define a \textit{projective algebraic set} determined by homogeneous polynomials $f_1, \dots, f_k\in \oF[x_0, \dots, x_D]$ as the common zero set of these polynomials.

One should think of the projective space as the ``completion" of the affine space, in which one adds certain points at infinity. Hence, to every affine algebraic set one can uniquely associate a projective variety which extends it. Formally, one can define a homogenization $f^h$ of the degree $d$ polynomial $f\in \oF[x_1, \dots, x_D]$ by setting $f^h(x_0, \dots, x_D)=x_0^df(x_1/x_0, \dots, x_D/x_0)$. Then, an affine algebraic set $V(f_1, \dots, f_k)\subseteq \oF^D$ extends to the projective algebraic set defined by polynomials $f_1^h, \dots, f_k^h$. 

Finite unions and arbitrary intersections of algebraic sets are also algebraic. An algebraic set $V$ is \textit{irreducible} if there are no algebraic sets $V_1, V_2\subsetneq V$ for which $V_1\cup V_2=V$, and an irreducible algebraic set is called a \textit{variety}. The decomposition theorem for algebraic sets states that there is a unique way to express an algebraic set as the finite union of varieties, which are then called the \textit{components} of the algebraic set (see Section 6 of Chapter 4 in \cite{CLO}).

Given an affine algebraic set $Y\subset \oF^D$, the \emph{ideal of $Y$}, denoted by $I(Y)$, is defined as the set of polynomials $f\in \oF[x_1, \dots, x_D]$ that vanish on $Y$. Clearly, if $g,h\in \oF[x_1, \dots, x_D]$ are polynomials, then the restrictions of $g$ and $h$ on $Y$ are identical if and only if  $g-h\in I(Y)$. Therefore, the space of polynomial functions on $Y$ is isomorphic to the quotient $\oF[x_1, \dots, x_D]/I(Y)$. One can analogously define the ideal of a projective algebraic set $Y\subseteq \bP^D(\oF)$, denoted by $I(Y)$, as the set of all homogeneous polynomials in $\oF[x_0, \dots, x_D]$ vanishing on $Y$.

Next, we define the dimension and the degree of a variety. We give two definitions of the dimension of a variety (for the proof of their equivalence, see Theorem I.7.5 in \cite{H}). Both of these definitions will be useful in our proofs. According to the first definition, the \textit{dimension} of the variety $V$ is the maximum value of $r$ for which there exists a chain of varieties $V_0, V_1, \dots, V_r$ such that $\emptyset=V_0\subsetneq V_1\subsetneq \dots\subsetneq V_r\subsetneq V$. 

Next, we give a definition the dimension based on the ideal $I(V)$. Fix a projective variety $V\subseteq \bP^D(\oF)$ and let $\oF_{t}[x_0, x_1, \dots, x_D]$ be the space of homogeneous polynomials of degree $t$, which is a finite-dimensional vector space over $\oF$. Further, let $I_{t}=I(V)\cap \oF_{t}[x_0, x_1, \dots, x_D]$ be the space of those homogeneous polynomials in $I(V)$ which have degree $t$. Since $I_{t}$ is a subspace of $\oF_{t}[x_0, x_1, \dots, x_D]$, one may consider the dimension of the quotient $\dim \big(\oF_{t}[x_0, x_1, \dots, x_D]/I_{t}\big)$, which is called the \textit{Hilbert function} of $I$. By Hilbert's theorem, the function $t\mapsto\dim \big(\oF_{t}[x_0, \dots, x_D]/I_{t}\big)$ is a polynomial for $t$ sufficiently large and it is called the \emph{Hilbert polynomial} of $I(V)$ (see e.g. Proposition 3 on p. 487. in \cite{CLO}). The \emph{dimension} of the variety $V$, denoted by $\dim(V)$, is then the degree of the Hilbert polynomial associated to $I(V)$. Furthermore, the \emph{degree} of the variety, denoted by $\deg V$, is defined as $(\dim V)!$ times the leading coefficient of this polynomials. Finally, the dimension and the degree of the affine variety are simply the dimension and the degree of the associated projective variety.

To get some intuition for these concepts, we remark that if $V=V(f)$ is a variety in $D$-dimensional space defined by a single irreducible polynomial $f$, then the dimension of $V$ is $D-1$ and the degree of $V$ is $\deg f$.

To conclude this section, let us mention two important results that we use in our proofs. The first one concerns the intersections of varieties whose degree and dimension are known and holds for affine and projective varieties alike (for a reference, see Example 8.4.6 in \cite{F}).

\begin{theorem}\label{thm:preliminary}
Let $V_1$ and $V_2$ be algebraic varieties of dimensions $d_1\leq d_2$ in a projective or affine ambient space of dimension $D$ such that $V_1\not\subseteq V_2$. If $Z_1, \dots, Z_k$ are the irreducible components of $V_1\cap V_2$, then $\dim(Z_i)\leq d_1-1$ for all $i$ and $\sum_{i=1}^k \deg(Z_i)\leq \deg V_1\cdot \deg V_2$. 
\end{theorem}

The second statement we use is a uniform version of Hilbert's theorem, which can be used to bound the Hilbert function (for a reference, see e.g. Chapter 9 of \cite{NA}). 

\begin{theorem}\label{thm:uniform_hilbert}
Let $V$ be a projective variety and let $I=I(V)$ be the associated ideal. Then, for every positive integer $t$,  $\dim \big(\oF_{t}[x_0, x_1, \dots, x_D]/I_{t}\big)\leq \deg(V) t^{\dim V}+\dim V.$
\end{theorem}

\subsection{Proofs of upper bounds via induced Tur\'an problems}

In this section, we prove Theorem \ref{thm:varieties upper} through the following approach. We construct a bipartite graph $H_{d, \Delta}$ with the property that $G(\cP, \cV)$ does not contain an induced copy of $H_{d, \Delta}$. Then, we show that any $K_{s, s}$-free bipartite graph with sides of size $m, n$ which does not contain an induced copy of $H_{d, \Delta}$ has at most $O(m^{\frac{d}{d+1}} n)$ edges.

Let us begin by describing the forbidden induced bipartite subgraph $H=H_{d, \Delta}$. Let $A$ and $B$ be the parts of $H_{d,\Delta}$, and $k=2^{\Delta^d}+1$. Part $B$ consists of $d+1$ ``layers" of vertices, which we describe as follows. The first two layers contain one vertex each, denoted by $v_1$ and $v_2$, respectively. For $3\leq \ell\leq d+1$, the $\ell$-th layer has $k^{\ell-2}$ vertices, denoted by $v_\ell^{(i_3, \dots, i_\ell)}$, where $i_3, \dots, i_\ell\in [k]$.

Next, we describe part $A$. For each $3\leq \ell\leq d+1$ and each sequence $(i_3, \dots, i_\ell)\in [k]^{\ell-2}$, we add $k$ vertices $w_{\ell,1}^{(i_3,\dots,i_{\ell})},\dots,w_{\ell,k}^{(i_3,\dots,i_{\ell})}$ to $A$ whose neighbours are $v_1, v_2, v_3^{(i_3)}, v_4^{(i_3, i_4)}, \dots, v_\ell^{(i_3, \dots, i_\ell)}$. Note that every vertex of $A$ has degree at most $d+1$. Furthermore, observe that vertices $v_{t}^{(i_3, \dots, i_t)}$ and $v_r^{(j_3, \dots, j_r)}$ have a common neighbour if and only if one of the sequences $(i_3,\dots, i_t)$ and $(j_3, \dots, j_r)$ is a prefix of the other.

\begin{proposition}\label{prop:forbidden induced subgraph}
If $\cP$ is a set of points and $\cV$ is a set of $d$-dimensional varieties of degree at most $\Delta$, the incidence graph $G(\cP, \cV)$ does not contain $H_{d, \Delta}$ as an induced subgraph, where the vertices of $A$ correspond to points and the vertices of $B$ correspond to varieties.
\end{proposition}
\begin{proof}
Assume that $G(\cP,\cV)$ contains an induced copy of $H_{d,\Delta}$ and let  $V_\ell^{(i_3, \dots, i_\ell)}$ be the variety corresponding to the vertex $v_\ell^{(i_3, \dots, i_\ell)}$. We show by induction on $2\leq \ell\leq d+1$ that one can choose a sequence $i_3,\dots,i_{\ell}$  with the following property. Let $Z_1, \dots, Z_t$ be the irreducible components of the intersection $V_1\cap V_2\cap\dots\cap V_\ell^{(i_3, \dots, i_\ell)}$ that satisfy $\dim(Z_i)\leq d-\ell+1$. Then the union $\bigcup_{i=1}^t Z_i$ contains all points corresponding to the common neighbours of $v_1, \dots, v_\ell^{(i_3, \dots, i_\ell)}$ in $A$.

For $\ell=2$, this statement is obvious, since all components of $V_1\cap V_2$ have dimension at most $d-1$ by Theorem~\ref{thm:preliminary}, noting that $V_1$ and $V_2$ are distinct varieties of dimension $d$. To show the inductive step, suppose we found a sequence $i_3,\dots,i_{\ell}$  satisfying the above property. Let $Z_1, \dots, Z_t$ be the irreducible components of the intersection of these varieties satisfying $\dim(Z_i)\leq d-\ell+1$. By iterated application of Theorem~\ref{thm:preliminary}, one can deduce that $t\leq \Delta^\ell$.

For each variety $V_{\ell+1}^{(i_3, \dots, i_\ell, j)}$, $j\in \{1, \dots, k\}$, let us denote by $T_j$ the set of components among $Z_1, \dots, Z_t$ which are fully contained in $V_{\ell+1}^{(i_3, \dots, i_\ell, j)}$. As $k=2^{\Delta^d}> 2^t$, there exist two sets $T_{j}$ and $T_{j'}$ which are identical. In other words, the varieties $V_{\ell+1}^{(i_3, \dots, i_\ell, j)}$ and $V_{\ell+1}^{(i_3, \dots, i_\ell, j')}$ contain the exact same components $Z_i$. We claim that the sequence $i_3,\dots,i_{\ell},j$ then satisfies the required conditions and suffices to perform the inductive step.

To see this, note that $V_{\ell+1}^{(i_3, \dots, i_\ell, j)}$ intersects all components $Z_i\notin T_j$ in subvarieties of dimension at most $d-\ell$ (or in the empty set). Thus, to complete the inductive step, it suffices to show that all common neighbours of $v_1, \dots, v_{\ell+1}^{(i_3, \dots, i_\ell, j)}$ are contained in the components $Z_i\notin T_j$. If there is a common neighbour $P$ of these vertices in a component $Z_i\in T_j$, it belongs to the variety $V_{\ell+1}^{(i_3, \dots, i_\ell, j')}$, and thus it is connected to the vertex $v_{\ell+1}^{(i_3, \dots, i_\ell, j')}$. But this is impossible, since $v_{\ell+1}^{(i_3, \dots, i_\ell, j)}$ and $v_{\ell+1}^{(i_3, \dots, i_\ell, j')}$ have no common neighbours in $H_{d,\Delta}$.

Taking $\ell=d+1$, we get a sequence $i_3,\dots,i_{d+1}$ such that $V_1,\dots, V_{d+1}^{(i_3,\dots,i_{d+1})}$ has the following property. If $Z_1,\dots,Z_t$ are the $0$-dimensional components of  $V_1\cap\dots\cap V_{d+1}^{(i_3,\dots,i_{d+1})}$, then $\bigcup_{i=1}^{t}Z_i$ contains $k$ points, corresponding to the common neighbours of $v_1,\dots,v_{d+1}^{(i_3,\dots,i_{d+1})}$. However, note that an iterated application of Theorem~\ref{thm:preliminary} implies $t\leq \Delta^{d+1}$ and so $t<k$. But this is a contradiction, since each $Z_i$ is a single point.
\end{proof}

Next, given a bipartite graph $H$, we prove a general upper bound on the number of edges in a $K_{s,s}$-free and induced $H$-free bipartite graph. A  similar statement was proved recently by Hunter and the authors of this paper \cite{HMST}, but their result is not directly applicable in case the host graph is bipartite with parts of unequal sizes. Fortunately, we can reuse most of the key auxiliary lemmas to adapt to this scenario.

Let $H=(A,B,E)$ be a bipartite graph such that every vertex in $A$ has degree at most $k$. Given a graph $G$, we say that a set of vertices $S\subset V(G)$ is $(H,k,s)$-rich if for every $T\subset S$, $|T|\leq k$, we have 
$$|\{v\in V(G)\backslash S:N(v)\cap S=T\}|\geq (4|A|s)^{|A|}.$$
We need the following two results from \cite{HMST}.

\begin{lemma}[Lemma 2.3 in \cite{HMST}]\label{lemma:HMST2.3}
    Let $G$ be a graph not containing $K_{s,s}$. If $G$ contains an $(H,k,s)$-rich independent set $S$ of size $|B|$, then $G$ contains $H$ as an induced subgraph in which $B$ is embedded into $S$.
\end{lemma}

\begin{proposition}[Proposition 2.4 in \cite{HMST}]\label{prop:HMST2.4}
    Let $G$ be a $K_{s,s}$-free graph and $X\subseteq V(G)$ a set of at least $(4|A||B| s)^{4|B|+10}$ vertices in which every $k$-tuple of vertices has at least $(4|A||B| s)^{2 |V(H)|}$ common neighbours. Then $X$ contains an $(H,k,s)$-rich independent set of size $|B|$. 
\end{proposition}

Furthermore, we use the following simple statement about independent sets in hypergraphs.

\begin{proposition}\label{prop:hyp_ind}
    Let $\cH$ be a $k$-uniform hypergraph with $N$ vertices and $M$ edges. Then $\cH$ contains an independent set of size at least $N^{\frac{k}{k-1}}/(4(M+N)^{1/(k-1)})$.
\end{proposition}

\begin{proof}
Let $p=(N/2(M+N))^{1/(k-1)}$. Let $X\subseteq V(\cH)$ be a random sample in which each vertex is included independently with probability $p$. Then $\cH[X]$ contains an independent set of size at least $|X|-e(\cH[X])$, as we can remove a single vertex from every edge to get an independent set. But $\mathbb{E}(|X|-e(\cH[X]))=pN-p^k M\geq \frac{pN}{2}$, so there is a choice for $X$ such that $|X|-e(\cH[X])\geq pN/2\geq N^{\frac{k}{k-1}}/(4(M+N)^{1/(k-1)})$.
\end{proof}

Now we are ready to prove our bound, which is a simple combination of the previous two results and the dependant random choice method \cite{FS}.

\begin{lemma}\label{lemma:induced_free}
Let $G=(U,V;E)$ be a bipartite graph, $|U|=m$ and $|V|=n$, and let $H=(A,B;E)$ be a bipartite graph such that every vertex in $A$ has degree at most $k\geq 2$. If $G$ contains no induced copy of $H$ in which $A$ is embedded to $U$, and $G$ is $K_{s,s}$-free, then
$$E(G)=O_{H,s}\left(m^{\frac{k-1}{k}}n+m\right).$$
\end{lemma}

\begin{proof}
Assume $e(G)\geq K(m^{\frac{k-1}{k}}n+m)$, where $K=K(H,s)$ is specified later. Let $t=(4|A||B| s)^{2 |V(H)|}$ and $z=(4|A||B| s)^{4|B|+10}$ be constants depending only on $H$ and $s$. We show that if $G$ contains no $K_{s,s}$, then $G$ contains an induced copy of $H$ in which $A$ is embedded to $U$. 

Let $u$ be a vertex in $U$, and let $X_0=N_G(u)$. Define the $k$-uniform hypergraph $\mathcal{H}$ on vertex set $X_0$ such that a $k$-element set $D\subset X_0$ is an edge if $D$ has at most $t$ common neighbours in $G$. First, we show that there is a choice for $u$ such that $\cH$ contains an independent set of size at least $z$. By Proposition \ref{prop:hyp_ind}, there is independent set of size at least $|X_0|^{\frac{k}{k-1}}/(4(|X_0|+e(\cH))^{\frac{1}{k-1}})$. Hence, it is enough to show that there exists a choice of  $u\in U$ for which $|X_0|^{k}>(4z)^{k-1} (|X_0|+e(\cH))$.

Choose $u$ uniformly at random from $U$. Then $\bE[|X_0|]=\frac{e(G)}{m}\geq K \frac{n}{m^{1/k}}+K$. On the other hand, for any fixed $k$ element set $D\subset V$ with less than $t$ common neighbours, we have $\mathbb{P}(D\subset X_0)\leq \frac{t}{m}.$ Since there are at most $n^{k}$ such $k$ element sets in $V$, we conclude that $\bE[e(\cH)]\leq \frac{t n^{k}}{m}.$ By Jensen's inequality, we have $\bE[|X_0|^{k}]\geq \bE[|X_0|]^{k}\geq K^{k} \frac{n^{k}}{m}+K^k$, and therefore $$\bE[|X_0|^{k}-(4z)^{k-1}(e(\cH)+|X_0|)]>0$$ when $K$ is sufficiently large. Hence, there exists a choice of $u$ for which $\cH$ contains an independent set of size $z$. Fix such a choice, and let $X\subset X_0$ be such an independent set.

As every $k$ element subset of $X$ has at least $t$ common neighbors, we can apply Proposition \ref{prop:HMST2.4} to find an $(H,k,s)$-rich set $S\subset X$ of size $|B|$. But then, by Lemma~\ref{lemma:HMST2.3}, there is an embedding of $H$ in which $B$ is embedded to $S$, finishing the proof.
\end{proof}

\begin{proof}[Proof of Theorem \ref{thm:varieties upper}]
By Proposition \ref{prop:forbidden induced subgraph}, the incidence graph $G(\cP,\cV)$ does not contain an induced copy of the bipartite graph $H_{d,\Delta}$, where part $A$ is corresponds to points. Also, every vertex of $A$ has degree at most $d+1$, so applying Lemma \ref{lemma:induced_free} with $k=d+1$ and $H=H_{d,\Delta}$ gives the desired bound.
\end{proof}

\subsection{Proofs of upper bounds via VC-dimension}

In this section, we show a variant of Theorem~\ref{thm:varieties upper} using a different approach. Note that we use a slightly stronger assumption that the varieties have bounded description complexity, instead of just bounded degree. However, since the description complexity can be bounded by a function of the degree and the ambient dimension, this assumption does not make a big difference (see e.g. Theorem 2.1.16 in \cite{T}). 

We say that the variety $V$ has \textit{description complexity} at most $t$ if it can be defined using at most $t$ polynomials $f_1, \dots, f_t\in \bF[x_1, \dots, x_D]$ of degree at most $t$.

\begin{theorem}\label{thm:varieties upper variant}
Let $\cP$ be a set of $m$ points in $\bF^D$ and let $\cV$ be a set of $n$ varieties in $\bF^D$, each of dimension $d$ and description complexity at most $t$. If the incidence graph $G(\cP, \cV)$ is $K_{s, s}$-free, then \[I(\cP, \cV)\leq O_{D,t, s}(m^{\frac{d}{d+1}} n+m).\]
\end{theorem}

Let us prepare the proof of this theorem. First, we recall some basic notions from the theory of VC-dimesion. Let $X$ be a ground set and let $\cF$ be a family of subsets of $X$. The \textit{shatter function} of the system $\cF$, denoted by $\pi_\cF(k)$, is defined as the maximum number of distinct intersections of a $k$-element set with members of $\cF$. Formally,
$$\pi_\cF(k)=\max_{A\subset X, |A|=k} |\{A\cap B:B\in \cF\}|.$$
For a sequence of varieties $V_1,\dots,V_k$ and a point $x$ in $\bF^D$, we define the \textit{containment pattern} of $V_1,\dots,V_k$ at $x$ as the set of indices $I\subseteq [k]$ of those varieties $V_i$ which contain $x$, that is $I=\{i\in [k]:x\in V_i\}$. Our first key lemma is an extension of a celebrated result of R\'onyai, Babai and Ganapathy \cite{RBG}  on the number of zero-patterns of polynomials. This statement, which might be of independent interest, implies that the family of possible containment-patterns has a polynomial shatter function.

\begin{lemma}\label{lemma:variety_zero_pattern}
Let $V_1,\dots,V_k$ be a sequence of varieties in $\bF^D$, each of dimension $d$ and description complexity at most $t$. The number of distinct containment patterns of $V_1,\dots,V_k$ is at most $O_{D, t}(k^{d+1})$.
\end{lemma}

The proof of Lemma~\ref{lemma:variety_zero_pattern} combines the ideas from \cite{RBG} with Hilbert's theorem. Hence, before presenting the proof, let us recall the result of R\'onyai, Babai and Ganapathy \cite{RBG}. Using our terminology, given sequence of polynomials $f_1, \dots, f_k\in \bF[x_1, \dots, x_D]$ and point $x\in \bF^D$, the \textit{zero-pattern} of $f_1,\dots,f_k$ at $x$ is the set of indices $I\subset [k]$ defined as $I=\{i\in [k]:f_i(x)=0\}$. If $\mathcal{Z}(f_1, \dots, f_k)$ is the set of all zero-patterns of $f_1, \dots, f_k$ and $f_1, \dots, f_k$ are polynomials of degree at most $\Delta$, then the number of distinct zero-patterns is at most 
\begin{equation}\label{eqn:zero-patterns polynomials}
    |\mathcal{Z}(f_1,\dots, f_k)|\leq  \binom{k\Delta}{D}.
\end{equation}
Note that (\ref{eqn:zero-patterns polynomials}) can be used as a black box to derive a weaker version of Lemma~\ref{lemma:variety_zero_pattern}, since every containment patterns of $V_1, \dots, V_k$ corresponds to a zero-patterns of their defining polynomials. However, this argument would lead to the weaker upper bound $O_t(k^D)$ in Lemma \ref{lemma:variety_zero_pattern}.

\begin{proof}[Proof of Lemma~\ref{lemma:variety_zero_pattern}.]
We work in the algebraic closure $\oF$ of the field $\bF$, which can only increase the number of containment patterns. We show that for every $a\in [k]$, the number of containment patterns $I\subset [k]$ with $a\in I$, i.e. the patterns coming from points $x\in V_a$, is at most $O_{t,D}(k^d)$. Then we are clearly done. Without loss of generality, let $a=1$, and let $M$ be the number of containment patterns $I\subseteq [k]$ with $1\in I$.

For $i=2,\dots,k$, let the defining polynomials of the variety $V_i$ be $f_{i,1}, \dots, f_{i,t}\in \oF[x_1, \dots, x_D]$, where we may repeat some of the polynomials to ensure that we have exactly $t$ polynomials. Since every containment pattern corresponds to a unique zero-pattern of polynomials $f_{2,1}, \dots, f_{k,t}$, we have $M\leq N$, where $N$ is the number of zero-patterns of $f_{2,1}, \dots, f_{k,t}$ on the variety $V_1$. Let $\widetilde{f}_{i,j}$ be the polynomial function induced on $V_1$ by the polynomial $f_{i, j}$, i.e. $\widetilde{f}_{i,j}=f_{i,j}\mod I(V_1)$, where we recall that $I(V_1)$ is the ideal of the variety $V_1$ (see Section~\ref{sec:alg geo background}).

Let $x_1, \dots, x_N\in V_1$ be points witnessing the $N$ distinct zero-patterns, and for $j\in [N]$, let $S_j\subset \{\widetilde{f}_{2,1},\dots,\widetilde{f}_{k,t}\}$ be the set of polynomial $\widetilde{f}$ for which $\widetilde{f}(x_j)\neq 0$.  Define the polynomial $$g_j(x)=\prod_{\widetilde{f}\in S_j} \widetilde{f}(x).$$ 
Clearly, we have $g_{j}(x_j)\neq 0$, and $$\deg g_j\leq \sum_{i,\ell} \deg \widetilde{f}_{i,\ell} \leq kt^2.$$ Furthermore, note that $g_{j}(x_{\ell})=0$ if $S_{j}\not\subset S_{\ell}$, since for every $\widetilde{f}\in S_{j}\setminus S_{\ell}$, we have $\widetilde{f}(x_{\ell})=0$.

We now argue that the polynomials $g_1,\dots,g_N$ are linearly independent. Otherwise, there exist scalars $\lambda_1, \dots, \lambda_N\in \oF$, not all zero, for which  \[\sum_{j=1}^n \lambda_j g_j=0.\] Choose $\ell$ such that $\lambda_{\ell}\neq 0$ and $|S_{\ell}|$ is minimal. If we plug  $x=x_{\ell}$ into the above equation, all terms except $g_{\ell}(x_{\ell})$ vanish. Indeed, for all $j\neq \ell$ we either have $\lambda_j=0$ or $S_j\not\subset S_{\ell}$, meaning that $g_j(x_{\ell})=0$. But $g_{\ell}(x_{\ell})\neq 0$, which is a contradiction. Thus, the polynomials $g_1, \dots, g_N$ are linearly independent. But the dimension of the space of polynomials of degree at most $kt^2$ on the variety $V_1$ is bounded by $O_{t, D}(k^{d})$ by Theorem~\ref{thm:uniform_hilbert}, and so $M\leq N\leq O_{t, D}(k^d)$. 
\end{proof}

In the proof of Theorem \ref{thm:varieties upper variant}, we also use the following result of \cite{FPSSZ}, whose proof follows from the Haussler packing lemma \cite{Haussler}. 

\begin{theorem}[Theorem 2.1 from \cite{FPSSZ}]\label{thm:zarankiewicz bounded VC}
Let $c>0$ and $r,s\in \mathbb{N}$, then there exists $c_1=c_1(c,D,s)>0$ such that the following holds. Let $G=(A, B, E)$ be a bipartite graph with $|A|=m$ and $|B|=n$ such that the set system $\cF_1=\{N(q):q\in A\}$  satisfies $\pi_{\cF_1}(k)\leq ck^r$ for every positive integer $k$. Then, if $G$ is $K_{s, s}$-free, we have 
\[|E(G)|\leq c_1(nm^{1-1/r} +m).\]
\end{theorem}

\begin{proof}[Proof of Theorem \ref{thm:varieties upper variant}]
Let $\cV=\{V_1,\dots,V_n\}$, and let $\cF$ be the set of containment patterns of $V_1,\dots,V_n$. Then $\pi_{\cF}(k)=O_{D,t}(k^{d+1})$ by Lemma \ref{lemma:variety_zero_pattern}. For every $x\in \cP$, the neighborhood $N(x)$ corresponds to the containment pattern $I_x=\{i\in [n]:x\in V_i\}$. Therefore, if $\cF_1=\{N(x):x\in \cP\}$, then $\pi_{\cF_1}(k)\leq \pi_{\cF}(k)=O_{D,t}(k^{d+1})$. Applying Theorem \ref{thm:zarankiewicz bounded VC} with $r=d+1$, we get $$I(\cP,\cV)=O_{D,t,s}(n m^{\frac{d}{d+1}}+m).$$
\end{proof}

\section{Unit distances}\label{sec:unit distances}

\subsection{Geometry of spheres over arbitrary fields}

In this section, we discuss basic properties of spheres in finite fields, highlighting some differences with the real space. Throughout this section, we consider a non-degenerate bilinear form $\langle \cdot, \cdot \rangle:\bF^d\to \bF$, with respect to which we define the notions of orthogonality and distance. The form is non-degenerate if for every $v\neq 0$ there exists some $w$ such that  $\langle v, w \rangle\neq 0$. It is standard to associate a norm to a bilinear form by defining $\|v\|^2=\langle v, v\rangle$. In what follows, a \emph{unit sphere} with center $w$ is defined to be the set of points $x\in \bF^d$ for which $\|x-w\|^2=1$.

Furthermore, we say that vectors $u, v\in \bF^d$ are \textit{orthogonal} if $\langle u, v\rangle=0$. One can generalize the notion to affine flats and say that $U, V\subseteq \bF^d$ are \textit{orthogonal} if for all points $u, u'\in U$ and $v, v'\in V$ one has $\langle u-u', v-v'\rangle =0$. A key property of orthogonal subspaces we use is that they satisfy $\dim U+\dim V\leq d$. Furthermore, for vectors $v_1, \dots, v_k\in \bF^d$, we denote the \textit{affine span} of $v_1, \dots, v_k$ by $\aff\{v_1, \dots, v_k\}$. More formally, we have $\aff\{v_1, \dots, v_k\}=\{\sum_{i=1}^d \lambda_i v_i | \lambda_i\in \bF_q, \lambda_1+\dots+\lambda_d=1\}$.

Finally, we define a specific bilinear form $\langle u, v\rangle_d$ and the associated norm $\|v\|_d$ on the space $\bF_q^d$ as follows. When $d\not\equiv 1\bmod 4$, we let $\langle u, v\rangle_d$ be the standard inner product defined by $\langle u, v\rangle_d=u_1v_1+\dots+u_dv_d$, while for $d\equiv 1\bmod 4$ we define $\langle u, v\rangle_d=u_1v_1+\dots+u_{d-1}v_{d-1}-u_dv_d$. In both cases, we define the norm $\|v\|_d^2=\langle v, v \rangle_d$.

\begin{lemma}\label{lemma:intersections of spheres}
Let $\langle \cdot, \cdot\rangle$ be a nondegenerate bilinear form on $\bF^d$ and $S_1, \dots, S_k$ be unit spheres in $\bF^d$, defined with respect to $\langle\cdot, \cdot\rangle$. Then there exists an affine subspace $U\subseteq \bF^d$ such that $S_1\cap \dots\cap S_k=S_1\cap U$. Moreover, $U$ is orthogonal to the affine flat spanned by the centers of $S_1, \dots, S_k$.
\end{lemma}
\begin{proof}
We prove the statement by induction on $k$. Let us denote the centers of the spheres $S_1, \dots, S_k$ by $w_1, \dots, w_k$.

When $k=1$, we may take $U=\bF^d$. For $k>1$, we assume by our induction hypothesis that $S_1\cap \dots\cap S_{k-1}=S_1\cap U$ for some affine subspace $U\subseteq \bF^d$, which is orthogonal to $\aff\{w_1, \dots, w_{k-1}\}$. Hence, we have $S_1\cap \dots\cap S_{k-1}\cap S_k=S_1\cap S_k\cap U$. Consider the defining equations for $S_1$ and $S_k$, which are \[\langle x-w_1, x-w_1\rangle=1 \text{ and } \langle x-w_k, x-w_k\rangle=1.\] By subtracting these two equations, we have $2\langle x, w_k-w_1\rangle+\langle w_1, w_1\rangle-\langle w_k, w_k\rangle=0$ for any $x\in S_1\cap S_k$. Note that this equation is linear in $x$ and hence it defines a hyperplane $H$, whose normal vector is $w_1-w_k$. We conclude that $S_1\cap S_k=S_1\cap H$ and therefore $S_1\cap \dots \cap S_k=S_1\cap S_k\cap U=S_1\cap (U\cap H)$, where $U\cap H$ is an affine space. 

Let us now argue that $U\cap H$ is orthogonal to $\aff\{w_1, \dots, w_k\}$. Let $u, u'\in U\cap H$ and $v, v'\in \aff\{w_1, \dots, w_k\}$ be arbitrary vectors. One may write $v-v'=\sum_{i=2}^k \lambda_i (w_i-w_1)$ with suitable $\lambda_2,\dots,\lambda_k\in \bF$. Therefore, we have 
\begin{align*}
\langle u-u', v-v'\rangle =\Big\langle u-u', \sum_{i=2}^{k-1} \lambda_i (w_i-w_1)\Big\rangle+\left\langle u-u', \lambda_k (w_k-w_1)\right\rangle.
\end{align*}
Note that the first term is zero since $u, u'\in U$, while the second term is zero because $u, u'\in H$. Hence, $U\cap H$ and $\aff\{w_1 \dots, w_k\}$ are orthogonal. 
\end{proof}

One counter-intuitive feature of spheres over finite fields is that they can contain flats. In the next lemma, we study some properties of such flats. 

\begin{lemma}\label{lemma:isotropic flats}
Let $\bF$ be a field of characteristic different from 2, let $V$ be a flat, and let $S$ be a unit sphere in $\bF^d$ centered at $w$. If $V\subseteq S$, then for any $x, y\in V$ we have $\langle x-y, x-y\rangle =0$ and $\langle x-w, x-y\rangle =0$.
\end{lemma}
\begin{proof}
If $x, y\in V$, then $x+\lambda (y-x)\in V$ for every $\lambda\in \bF$. Since $S=\{v\in \bF^d:\langle v-w, v-w\rangle=1\}$ and $V\subseteq S$, we must have 
\[1=\langle x+\lambda (y-x)-w, x+\lambda (y-x)-w\rangle=\langle x-w, x-w\rangle + 2\lambda \langle x-w, y-x\rangle +\lambda^2 \langle y-x, y-x\rangle.\]
Since $\langle x-w, x-w\rangle =1$, we obtain $\lambda^2 \langle y-x, y-x\rangle+2\lambda \langle x-w, y-x\rangle =0$ for all $\lambda$, implying that $\langle y-x, y-x\rangle=0$ and $\langle x-w, y-x\rangle=0$, just as claimed. 
\end{proof}

A flat $V$ for which $\langle x-y, x-y\rangle =0$ for all $x, y\in V$ is called \textit{totally isotropic}.

\begin{lemma}\label{lemma:auxiliary odd d}
Let $\bF$ be a field containing no roots of the equation $x^2=-1$ and let $d=2k+1$. Furthermore, consider the bilinear form $\langle\cdot, \cdot\rangle_d$ on $\bF^d$. Then $\bF^d$ does not contain a pair $(V,w)$, where $V$ is a totally isotropic flat of dimension $k$, and $w$ is a vector of norm $1$ orthogonal to $V$.
\end{lemma}
\begin{proof}
Assume that $\bF^d$ contains such a pair $(V,w)$. After translation, we may assume that $V$ contains the origin. Let $v_1, \dots, v_k$ be a basis of the subspace $V$ which satisfies $\langle v_i, e_j\rangle_d =\delta_{ij}$ for all $i, j\in [k]$, where $e_1, \dots, e_d$ are the standard basis vectors. Here, $\delta_{ij}$ is defined to be $1$ when $i=j$ and $0$ otherwise. After possibly rearranging the coordinates, Gaussian elimination shows that it always possible to find such a basis. Furthermore, since $V$ is totally isotropic, any vector $w'=w+v$ with $v\in V$ is also of unit norm and orthogonal to $V$. Therefore, we may subtract an appropriate linear combination of $v_1, \dots, v_k$ from $w$ and assume that $\langle w, w\rangle_d =1$ and $\langle w, e_j\rangle_d=0$ for all $j\in [k]$.

To derive a contradiction, we compute the Gram matrix $G$ of the vectors $v_1'=v_1-e_1, \dots, v_k'=v_k-e_k, w$. Note that the first $k$ coordinates of each of these $k+1$ vectors are equal to $0$ and therefore one might consider $v_1', \dots, v_k', w$ as vectors in $\bF^{k+1}$ without changing their Gram matrix. Hence, we have $G(i,j)=\langle v_i', v_j'\rangle_d =\langle v_i, v_j\rangle_d -\langle e_i, v_j\rangle_d -\langle v_i, e_j\rangle_d +\langle e_i, e_j\rangle_d=-\delta_{ij}$ for all $i,j\in [k]$. Similarly, one can compute $G(i,k+1)=G(k+1,i)=\langle v_i', w\rangle_d=\langle v_i, w\rangle_d -\langle e_i, w\rangle_d =0$ and $G(k+1,k+1)=\langle w, w\rangle_d=1$. In conclusion, $G=\diag[-1,\dots,-1,1]$. 

On the other hand, let $A$ be a $(k+1)\times (k+1)$ the matrix whose columns are the vectors $v_1', \dots, v_k', w$ (recall that we may consider these vectors as elements of $\bF^{k+1}$). Recalling the definition of the the bilinear form $\langle v_i', v_j'\rangle_d$, we note that we can rewrite it as $\langle v_i', v_j'\rangle_d=(v_i')^T M v_j'$, where $M=I_{k+1}$ if $d\equiv 3\bmod 4$ and $M=\diag[1, \dots, 1, -1]$ if $d\equiv 1\bmod 4$. Hence, one can express the Gram matrix as $G=A^T M A$ and so $A^TMA=\diag[-1, \dots, -1, 1]$. By taking determinants, we see that $\det(A)^2\det(M)=\det(\diag[-1, \dots, -1, 1])=(-1)^k$. 

The final observation is that $\det(M)=-1$ when $d\equiv 1\bmod 4$ and $\det(M)=1$ otherwise, meaning that $\det(M)=(-1)^{k-1}$. Hence, no matter the parity of $k$, we have $\det(A)^2=-1$. But this is impossible, since the equation $x^2=-1$ has no solution in $\bF$. 
\end{proof}

\subsection{Proofs of the upper bounds}

In this section, we prove Theorem \ref{thm:unit_distance}. In particular, we prove the following incidence bound between points and unit spheres defined with respect to any nondegenerate bilinear form, from which the theorem immediately follows.

\begin{proposition}\label{prop:point-sphere incidences}
Let $\cP$ be a set of $n$ points  and let $\cS$ be a collection of $n$ unit spheres in $\bF^d$. If the incidence graph $G(\cP, \cS)$ is $K_{s, s}$-free, then the number of incidences between $\cP$ and $\cS$ is at most $O_{d,s}(n^{2-\frac{1}{\lceil d/2\rceil +1}})$.
\end{proposition}

The proof follows the ideas from Section~\ref{sec:illustration}, where we showed that incidence graphs of points and hyperplanes in $\bF^d$ avoid the pattern $\Pi_d$. Here, we will show that the point-sphere incidence graphs in $\bF^d$ avoid the pattern $\Pi_{d+1}$ and use this to deduce Proposition~\ref{prop:point-sphere incidences}.

\begin{lemma}\label{lemma:forbidden pattern}
Let $\cP$ be a set of $n$ points and let $\cS$ be a set of $n$ spheres in $\bF^d$. The incidence graph $G(\cP, \cS)$ does not contain the pattern $\Pi_{d+1}$ defined above.
\end{lemma}
\begin{proof}
For contradiction, let us assume that $G(\cP, \cS)$ does contain $\Pi_{d+1}$, whose vertices $a_1, \dots, a_{d+1}$ correspond to points $P_1, \dots, P_{d+1}$ and $b_1, \dots, b_{d+1}$ correspond to unit spheres $S_1, \dots, S_{d+1}$.

We  show by induction that $S_1\cap \dots \cap S_k=S_1\cap U$ for a affine flat $U$ of dimension at most $d-k+1$, for all $k\leq d+1$. For $k=1$ this statement is trivial. For $k\geq 2$, we may assume that $S_1\cap \dots\cap S_{k-1}=S_1\cap U$ for some flat $U$ of dimension at most $d-k+2$. By Lemma~\ref{lemma:intersections of spheres}, there exists a flat $H$ for which $S_1\cap S_k=S_1 \cap H$. Let $V=U \cap H$. Then
$S_1\cap \dots\cap S_{k}=S_1\cap U \cap S_k=S_1\cap (U \cap H)=S_1 \cap V$. The point $P_{k-2}$ is contained in the spheres $S_1, \dots, S_{k-1}$, but not in $S_{k}$. Therefore, $P_{k-2}$ is a point of $U$ not contained in $V$, implying that $\dim(V)\leq \dim(U)-1\leq d-k+1$.

Applying this claim with $k=d+1$, we conclude that there exists a $0$-dimensional affine space $U$ for which $S_1\cap \dots \cap S_{d+1}=S_1\cap U$. In other words, $S_1\cap \dots \cap S_{d+1}$ is a single point. Then, it is not possible for both $P_{d}$ and $P_{d+1}$ to be contained in $S_1\cap \dots \cap S_{d+1}$, presenting a contradiction.
\end{proof}

\begin{proposition}
Let $G$ be a bipartite graph on $2n$ vertices which does not contain $K_{s, s}$ or the pattern $\Pi_{d+1}$. Then, $G$ has at most $O_{d,s}\Big(n^{2-\frac{1}{\lceil d/2\rceil +1}}\Big)$ edges.
\end{proposition}
\begin{proof}
Lemma~\ref{lemma:main upper bounds} it is proved that a balanced bipartite graph on $2n$ vertices containing no $K_{s, s}$ or $\Pi_{d+1}$ has at most $O_{d,s}\Big(n^{2-\frac{1}{\lceil (d+2)/2\rceil }}\Big)$.
\end{proof}

\subsection{Constructions}

In this section, our goal is to prove Theorem \ref{thm:sphere_construction}. More precisely, for every $n$ and $d$, we construct a set $\cP$ of $n$ points in $\bF_q^d$, where $q=p^r$ is some prime power, with the properties that $\cP$ spans $\Omega(n^{2-\frac{1}{\lceil d/2\rceil+1}})$ unit distances and that the unit distance graph does not contain $K_{s, s}$. 

Our construction is based on so called variety-evasive sets. For fixed parameters $1\leq k\leq d$ and $\Delta, s>0$, we say that a set $U$ is $(k, s)$-variety-evasive if every variety $V\subseteq \bF_q^d$ of dimension $k$ and degree at most $\Delta$ intersects $U$ in less than $s$ points. Here we only consider varieties of degree at most $2$ and therefore we fix $\Delta=2$. Dvir, Koll\'ar and Lovett constructed large variety-evasive sets and in particular the following theorem is a special case of Corollary 6.1 from \cite{DKL}.

\begin{theorem}\label{thm:variety evasive sets}
For every positive integer $d$ and prime power $q$, there exist a constant $s=s(d)$ and set $U$ of $q^{d-k}$ points in $\bF_q^d$ such that any $k$-dimensional variety of degree at most $2$ intersects $U$ in less than $s$ points.
\end{theorem}

Let us now explain the idea behind our constructions. We set $k=\lfloor d/2\rfloor$ and choose $p$ to be a prime with $p\equiv 3\bmod 4$ and $p\approx n^{\frac{1}{\lceil d/2\rceil + 1}}$. Then, we use Theorem~\ref{thm:variety evasive sets} to find a $(k-1, s)$-variety evasive set $U\subseteq \bF_p^d$ of size $p^{d-k+1}=p^{\lceil d/2\rceil + 1}$. Since $U$ is $(k-1, s)$-variety evasive, we are able to show that the unit distance graph does not contain $K_{s, s}$ as a subgraph. However, in order to ensure many unit distances, we need to take the union of $U$ with a random shift, i.e. we set $\cP_x=U\cup (U+x)$ where $x\in \bF_p^d$ is chosen uniformly at random. 

\begin{proposition}\label{prop:many unit distances}
There exists $x\in \bF_p^d$ for which $\cP_x$ spans $\Omega\Big(|\cP_x|^{2-\frac{1}{\lceil d/2\rceil+1}}\Big)$ unit distances with respect to the norm $\|\cdot \|_d$.
\end{proposition}
\begin{proof}
Choose $x$ form the uniform distribution on $\bF_p^d$. The number of unit distances spanned by $\cP_x$ is at least the number of pairs $u, v\in U$ such that $\|u-(v+x)\|_d=1$. For a fixed pair $u, v\in U$, the equation $\|u-v-x\|_d=1$ is quadratic in $x$ and therefore it has $(1+o(1))p^{d-1}$ solutions (see e.g. Theorems 6.26 and 6.27 in \cite{LN}). Hence, for fixed $u, v\in U$ we have $\Pb[\|u-(v+x)\|_d=1]=(1+o(1))p^{-1}$ and so the expected number of pairs $u, v\in U$ with $\|u-(v+x)\|_d=1$ is $(1+o(1))|U|^2/p$. Therefore, $x$ can be chosen such that the number of unit distances spanned by $\cP_x$ is $\Omega(|U|^2/p)=\Omega\Big(|\cP_x|^{2-\frac{1}{\lceil d/2\rceil+1}}\Big)$.
\end{proof}

Let us fix the value $x$ from Proposition~\ref{prop:many unit distances} from now on and write $\cP=\cP_x$. Note that the set $\cP$ is $(k-1, 2s)$-variety evasive, since no variety of dimension $k-1$ and degree at most $2$ can contain more than $s$ points of either $U$ or $U+x$. 

\begin{proposition}
The unit distance graph of $\cP$ with respect to the norm $\|\cdot \|_d$ does not contain $K_{4s, 4s}$ as a subgraph.
\end{proposition}
\begin{proof} 
We argue by contradiction, assuming that there exist points $v_1, \dots, v_{4s},w_1, \dots, w_{4s}\in \cP$ such that $\|v_i-w_j\|_d=1$ for all $i, j\in [4s]$. Let $V=\aff \{v_1, \dots, v_{4s}\}$ and $W=\aff \{w_{1}, \dots, w_{4s}\}$. 

We claim that the flats $V$ and $W$ are orthogonal, with respect to the bilinear form $\langle \cdot, \cdot\rangle_d$. If we denote by $S_1, \dots, S_{4s}$ the unit spheres centered at $w_{1}, \dots, w_{4s}$, the assumption implies that $v_{1}, \dots, v_{4s}\in \bigcap_{i=1}^{4s} S_i$. By Lemma~\ref{lemma:intersections of spheres}, there exists an affine flat $V'$, orthogonal to $W$, such that $\bigcap_{j=1}^{4s} S_j=S_1\cap V'$. Since $V\subseteq V'$, we conclude $V$ and $W$ are orthogonal, which implies that $\dim V+\dim W\leq d$. 

Since $\cP$ is $(k-1, 2s)$-variety evasive, no $(k-1)$-dimensional flat contains $2s$ points of $\cP$ (note that an affine flat is a variety of degree $1$). In particular, this means that $\dim V\geq k$ and $\dim  W\geq k$. Combining this with $\dim V+\dim W\leq d=2k+1$, we conclude that $\dim V = k$ or $\dim W = k$. By symmetry, we may assume that $\dim V=k$. 

We claim that $V\subseteq S_j$, for every $j\in [4s]$. Suppose this was not the case, i.e. that there was some $j$ for which $V\not\subseteq S_j$. Since $V$ is a variety of dimension $k$ and degree $1$ and $S_j$ is a variety of dimension $d-1$ and degree $2$, Theorem~\ref{thm:preliminary} implies that if $V\cap S_j$ is a union of irreducible components $Z_1, \dots, Z_t$, then $\sum_{i=1}^t \deg Z_i\leq 2$. Hence, $V\cap S_j$ consists either of one $(k-1)$-dimensional component of degree $2$ or of at most two $(k-1)$-dimensional components of degree $1$. However, recall that $v_1, \dots, v_{4s}\in S_j\cap V$ and that every $(k-1)$-dimensional variety of degree at most $2$ contains less than $2s$ points of $\cP$. Hence, in both cases we have a contradiction and so we must have $V\subseteq S_j$. Therefore, by Lemma~\ref{lemma:isotropic flats}, $V$ must be a totally isotropic flat and we have $\langle w_j-y, z-y\rangle_d=0$ for all $y, z\in V$. 

Since $\langle w_j-v_i, w_j-v_i\rangle_d =1$, the point $w_j$ does not lie in the flat $V$. Also, $\aff (\{w_j\}\cup V)$ is orthogonal to $V$, since $\langle w_{j}-x, y-x\rangle_d=0$ for all $x, y\in V$. Hence, $\dim \aff (\{w_{j}\}\cup V)=\dim V+1=k+1$. If $d$ is even, this is a contradiction since $\dim V+\dim \aff (\{w_{j}\}\cup V)=2k+1>d$.

If $d$ is odd, we found a totally isotropic $k$-flat with a vector $w_{j}-v_{i}$ which has norm $1$ and is orthogonal to this $k$-flat. But this is impossible by Lemma~\ref{lemma:auxiliary odd d}.
\end{proof}

Let us now put all the ingredients together and show the general construction.

\begin{proof}[Proof of Theorem \ref{thm:sphere_construction}]
Let $p$ be the smallest prime such that $p\equiv3 \bmod 4$ and $p>n^{\frac{1}{\lceil d/2\rceil +1}}$. By a theorem of Breusch \cite{B}, we have $p\leq 2n^{\frac{1}{\lceil d/2\rceil +1}}$, as long as $n^{\frac{1}{\lceil d/2\rceil +1}}\geq 7$. Let $\cP$ be the set of at least $p^{\lceil d/2\rceil +1} \geq n$ points constructed in this section. For $d\not\equiv 1\bmod 4$, taking a random $n$ element subset of $\cP$ gives the desired set. However, for $d\equiv 1 \bmod 4$ the standard norm and the norm $\|\cdot\|_d$ are different as we need one final step to complete the construction.

If $d\equiv 1\bmod 4$, we let $q=p^2$, we choose $\alpha\in \bF_q$ to be a solution of $\alpha^2=-1$, and consider the map $\varphi:\bF_p^d\to \bF_q^d$ given by $\varphi((x_1, \dots, x_d))=(x_1, \dots, x_{d-1}, \alpha x_d)$. Note that $\varphi(u)$ and $\varphi(v)$ are at unit distance measured with respect to the standard inner product if and only if $\|u-v\|_d=1$. Hence, we conclude that the set $\{\varphi(u)|u\in \cP\}\subseteq \bF_q^d$ spans $\Omega\Big(|\cP|^{2-\frac{1}{\lceil d/2\rceil+1}}\Big)$ unit distances and has no $K_{s, s}$ in the unit distance graph, where unit distances are measured with respect to the standard inner product. Now taking a random $n$ element subset of $\cP$ completes the proof.
\end{proof}

\section{The algebraic Zarankiewicz problem}\label{sec:algebraic zarankiewicz}

In this section, we present the proof of our bound on Zarankiewicz's problem in algebraic graphs of bounded complexity. We prove the upper and lower bounds separately, and then we use the lower bound to show Theorem~\ref{thm:construction varieties}. In the proof of the upper bound, we combine Theorem~\ref{thm:zarankiewicz bounded VC} with the estimate~(\ref{eqn:zero-patterns polynomials}). 

For the lower bounds, we use the random polynomial method, which was pioneered by Bukh \cite{Bukh}. The main idea is to choose the polynomial defining the algebraic graph to be a random low-degree polynomial, chosen uniformly among all polynomials of degree at most $\Delta$ in $\bF[x_1, \dots, x_D]$, where $\bF$ is a finite field of appropriate size. The key observation is that for any fixed $x_1, \dots, x_s\in \bF^D$ and random polynomial $f$ of degree at most $\Delta$, the random variables $\{f(x_j)| j\in [s]\}$ are independent and uniform in $\bF$ as long as $s\leq \min\{\Delta, |\bF|^{1/2}\}$ (see Claims 3.3 and 3.4 in \cite{ST23}).

\begin{theorem}\label{thm:algebraic zarankiewicz upper}
Let $G$ be an algebraic graph of description complexity at most $t$ on the vertex set $\cP\cup \cQ$, where $\cP\subseteq \bF^{d_1}$, $\cQ\subseteq \bF^{d_2}$ and $|\cP|=m$, $|\cQ|=n$. If $G$ is $K_{s, s}$-free, the number of edges in $G$ is at most $$O_{d_1,d_2,t, s}(\min\{m^{1-1/d_1}n, mn^{1-1/d_2}\}).$$
\end{theorem}
\begin{proof}
Let the graph $G$ be defined by the polynomials $f_1, \dots, f_t:\bF^{d_1}\times \bF^{d_2}\to \bF$ and a boolean formula $\Phi$, where $\deg f_i\leq t$ for $i=1, \dots, t$.

Since the roles of $m$ and $n$ are symmetric, it suffices to show that any $K_{s, s}$-free algebraic graph $G$ in $(\bF^{d_1}, \bF^{d_2})$ has at most $O_{t, s}(mn^{1-1/d_2}+n)$ edges. To do this, we use Theorem~\ref{thm:zarankiewicz bounded VC}. Define the set system $\cF=\{N(v):v\in \cQ\}$ on the ground set $\cP$. Our main goal is to show $\pi_{\cF}(k)\leq O_{t, d_2}(k^{d_2})$ for all $k$, since in this case Theorem~\ref{thm:zarankiewicz bounded VC} shows that $e(G)\leq O_{t, s}(mn^{1-1/d_2}+n)$. To this end, we fix  $k$ points $x_1, \dots, x_k\in \cP$ and bound the number of possible intersections $\{x_1, \dots, x_k\}\cap N(v)$ for $v\in \cQ$.  For every $i\in [k]$ and $j\in [t]$, consider the polynomial $y\mapsto f_j(x_i, y)$ over $\bF^{d_2}$ Let $F$ be the set of these $kt$ polynomials.

The number of distinct intersections $\{x_1, \dots, x_k\}\cap N(v)$ is bounded by the number of zero-patterns of polynomials in $F$, since the zero-pattern of polynomials in $F$ at a point $v\in \bF$ can be used to determine which $x_i$ is adjacent to $v$. Since $F$ is a set of $kt$ polynomials of degree at most $t$, by (\ref{eqn:zero-patterns polynomials}) we have that the number of zero-patterns of the polynomials in $F$ is bounded by $\binom{kt^2}{d_2}$. Therefore, $\pi_{\cF}(k)\leq \binom{kt^2}{d_2}= O_{t, d_2}(k^{d_2})$, finishing the proof.
\end{proof}

To show the lower bounds, we will need the following simple lemma.

\begin{lemma}\label{lemma:zeros of random polynomials}
Let $\Delta, D\geq 3$ be fixed integers and let $p\geq 4$ be a prime. Let $f\in \bF_p[x_1, \dots, x_D]$ be chosen uniformly at random among polynomials of degree at most $\Delta$, then $f$ has at least $p^{D-1}/2$ zeros in $\bF_p^D$ with probability at least $\frac{3}{4}$.
\end{lemma}
\begin{proof}
Since $f(x)$ is a uniformly distributed in $\bF_p$ for any fixed $x\in \bF_p^D$, the expected number of zeros of $f$ is $\bE\big[|V(f)|\big]=p^{D-1}$. On the other hand, using the second moment method, one can also show that the random variable $|V(f)|$ is concentrated around its mean. More precisely, if we denote by $\mathbf{1}_x$ the indicator random variable of $f(x)=0$, we have 
\[\Var\big(|V(f)|\big)= \sum_{x, y\in \bF^D} \Cov\big(\mathbf{1}_x, \mathbf{1}_y\big).\]
Claim 3.4 of \cite{ST23} states that $f(x)$ and $f(y)$ are independent random variables when $x\neq y$, if $p\geq 4$ and $\Delta\geq 2$. Moreover, for $x=y$, we have $\Cov\big(\mathbf{1}_x, \mathbf{1}_x\big)\leq \Pb[f(x)=0]=p^{-1}$. Hence, $\Var\big( |V(f)|\big)\leq p^{D-1}$. By Chebyshev's inequality, we have $\Pb[ |V(f)|<p^{D-1}-\lambda \sqrt{p^{D-1}}]\leq \frac{1}{\lambda^2}$, which for $\lambda=\frac{1}{2}\sqrt{p^{D-1}}$ gives 
\[\Pb\left[|V(f)|<\frac{p^{D-1}}{2}\right]\leq \frac{4}{p^{D-1}}\leq \frac{1}{4}.\]
\end{proof}

\begin{proposition}\label{prop:algebraic zarankiewicz lower}
For any $d_1, d_2$ and sufficiently large integers $m, n$, there exists a field $\bF$, sets of points $\cP\subseteq \bF^{d_1}, \cQ\subseteq \bF^{d_2}$, $|\cP|=m, |\cQ|=n$ and a polynomial $f:\bF^{d_1}\times \bF^{d_2}\to \bF$ of degree at most $\Delta=(d_1+d_2)^2$ with the following property. The algebraic graph $G$ defined on $\cP\cup \cQ$ by the equation $f(x, y)=0$ is $K_{s, s}$-free, where $s=d_1+d_2$, and has $\Omega(\min\{m^{1-1/d_1}n, mn^{1-1/d_2}\})$ edges.
\end{proposition}
Observe that the graph $G$ has description complexity at most $(d_1+d_2)^2$ and therefore Proposition~\ref{prop:algebraic zarankiewicz lower} suffices to show the second part of Theorem~\ref{thm:algebraic zarankiewicz}.

\begin{proof}
By symmetry, we may assume that $mn^{1-1/d_2}<m^{1-1/d_1}n$, which is equivalent to $m\leq n^{d_1/d_2}$. Let $p$ be the smallest prime larger than $n^{1/d_2}$, which satisfies $p<2n^{1/d_2}$ by Bertrand's postulate.

We define an algebraic graph $G_0$ in $(\bF_p^{d_1}, \bF_p^{d_2})$  as follows. Let $\cP_0=\bF_p^{d_1}$ and $\cQ_0=\bF_p^{d_2}$. Furthermore, let $f:\bF_p^{d_1}\times \bF_p^{d_2}\to \bF_p$ be a polynomial chosen randomly from the uniform distribution on all polynomials  of degree at most $\Delta=(d_1+d_2)^2$. Finally, set $x\in \cP_0$ and $y\in \cQ_0$ to be adjacent in $G_0$ if and only if $f(x, y)=0$.

We begin by showing that the probability $G_0$ contains $K_{s, s}$ is at most $\frac{1}{4}$. Given $2s$ points $x^{(1)}, \dots, x^{(s)}\in \cP_0, y^{(1)}, \dots, y^{(s)}\in \cQ_0$, the $s^2$ points $(x^{(i)}, y^{(j)})$ are independent uniform random variables as long as $s^2\leq \min\{\Delta, p^{1/2}\}$, see \cite{Bukh}. Therefore, the probability  that $f(x^{(i)}, y^{(j)})=0$ for all $i, j\in [s]$ is exactly $p^{-s^2}$.  Hence, we can apply the union bound over all subsets of $\cP_0$ and $\cQ_0$ of size $s$ to deduce
\[\Pb[K_{s, s}\subseteq G]\leq \binom{p^{d_1}}{s}\binom{p^{d_2}}{s} p^{-s^2}\leq \frac{p^{d_1s}}{s!}\frac{p^{d_2s}}{s!}p^{-s^2}\leq \frac{1}{(s!)^2}\leq \frac{1}{4}.\]

On the other hand, Lemma~\ref{lemma:zeros of random polynomials} shows that with probability at least $3/4$, the polynomial $f$ has at least ${p^{d_1+d_2-1}}/{2}$ zeros, which means that $G_0$ has at least  ${p^{d_1+d_2-1}}/{2}$ edges with probability at least $3/4$. Therefore, there exists a polynomial $f$ such that the graph $G_0$ is $K_{s, s}$-free and has at least $\frac{1}{2p}|\cP_0||\cQ_0|$ edges. Let us fix this graph $G_0$.

Then we get our final graph $G$ by sampling an $m$ element subset $\cP$ of $\cP_0$, and an $n$ element subset $\cQ$ of $\cQ_0$. Then, the expected number of edges of $G$ is $\frac{|\cP|}{|\cP_0|}\frac{|\cQ|}{|\cQ_0|}e(G_0)\geq \frac{mn}{2p}$. Hence, there exists a choice of $\cP, \cQ$ for which $e(G)\geq \frac{mn}{2p}$, let us fix this choice. This graph $G$ is clearly algebraic and $K_{s, s}$-free, since it is an induced subgraph of a $K_{s, s}$-free algebraic graph. Furthermore,  $e(G)\geq \frac{mn}{2p}\geq \frac{mn}{2} (2n)^{-1/d_2}=\Omega(mn^{1-1/d_2})$. This completes the proof. 
\end{proof}

We conclude the section by proving Theorem~\ref{thm:construction varieties}, which demonstrates that our upper bounds on point-variety incidences are tight.

\begin{proof}[Proof of Theorem \ref{thm:construction varieties}]
Consider the case $d=D-1$. The general case follows by embedding $\bF_p^{d+1}$ into a higher dimensional space. Let $p$ be a prime between $m^{1/D}$ and $2m^{1/D}$, which exists by Bertrand's postulate. Further, let $D'=\lceil \alpha D\rceil$, and let $G$ be an algebraic graph in $(\bF_p^D, \bF_p^{D'})$ with parts $\cP, \cQ$ of size $m, n$ obtained through Proposition~\ref{prop:algebraic zarankiewicz lower}. 

We construct a set of hypersurfaces $\cV$ from the set of points $\cQ$. For each point $q\in \cQ$, we consider the polynomial $f_q:\bF_p^{D}\to \bF_p$ given by $f_q(x)=f(x, q)$. Since $f_q(x)$ may not be irreducible, let $g_q(x)$ be an irreducible factor of $f_q(x)$ with the largest number of roots in $\cP$. Finally, define the algebraic set $V_q=\{x\in \bF_p^D|g_q(x)=0\}$. We show that the point set $\cP$ and the collection $\cV=\{V_q|q\in \cQ\}$ satisfy all conditions of the theorem.

Since $g_q(x)$ is an irreducible polynomial, the ideal $\langle g_q(x)\rangle \subseteq \bF_p[x_1, \dots, x_D]$ is a prime ideal, and therefore by Proposition 3, page 207 of \cite{CLO}, the algebraic set defined by $g_q(x)$ is irreducible and therefore a variety. Furthermore, the dimension of the variety $V_q=\{x\in \bF_p^D| g_q(x)=0\}$ is $D-1=d$. Hence, $\cV$ is indeed a collection of varieties of required dimension. 

Let us now argue that there are no distinct varieties $V_{q_1}, \dots, V_{q_s}\in \cV$ with at least $s$ points of $\cP$ in common, where $s=(D+D')^2$. If such $s$ varieties exist, the common neighbourhood of points $q_1, \dots, q_s$ in the graph $G$ contains at least $s$ points, which is not possible since $G$ is $K_{s, s}$-free. Hence, the incidence graph $G(\cP, \cV)$ is $K_{s, s}$-free.

Finally, we argue that there are many incidences between $\cP$ and $\cV$, i.e. that $I(\cP, \cV)\geq \Omega_{D, \alpha}(m^{\frac{D-1}{D}}n)$. The main observation is that the polynomials $f_q(x)$ have at most $(D+D')^2$ irreducible factors, since $\deg f_q(x)\leq (D+D')^2$. Hence, the number of roots of $g_q(x)$ in $\cP$ is at least $\frac{1}{(D+D')^2}$ times the number of roots of $f_q(x)$ in $\cP$. But the number of roots of $f_q(x)$ in $\cP$ summed over all $q\in \cQ$ is the number of edges of $G$ and therefore we have:
\begin{align*}
    I(\cP, \cV)=\sum_{q\in \cQ} |V_q\cap \cP|\geq \sum_{q\in \cQ} \frac{|\{f_q(x)=0\}\cap \cP|}{(D+D')^2}\geq \frac{e(G)}{(D+\lceil \alpha D\rceil )^2}\geq \Omega_{D, \alpha}\Big(\min\{m^{\frac{D-1}{D}} n, mn^{\frac{D'-1}{D'}}\}\Big).
\end{align*}
It is not hard to see that $m^{-1/D}\leq n^{-1/D'}$ since $m^{D'}=m^{\lceil \alpha D\rceil }\geq m^{\alpha D}=n^D$. Therefore, $I(\cP, \cV)\geq \Omega_{D, \alpha}\Big(m^{\frac{D-1}{D}} n\Big)$, which suffices to complete the proof.
\end{proof}

\section{Concluding remarks}

In this paper, we proposed a novel combinatorial approach to study incidence problems over arbitrary fields. We considered several natural problems, such as determining the maximum number of point-hyperplane incidences, point-variety incidences, and unit distances, under standard non-degeneracy assumptions.

Our bounds are also optimal in the sense that they cannot be improved unless assumptions on the underlying field are made. Hence, it would be interesting to decide whether the techniques in this paper can be combined with some field-dependent information to improve some of our bounds. Such improved bounds are known over the real numbers \cite{AS07,FPSSZ}, and also in some cases over finite fields \cite{BKT,Rudnev}. For example, Rudnev's theorem \cite{Rudnev} shows that when $n\leq p^2$,
there are at most $O(m\sqrt{n}+ms)$ incidences between $m$ points and $n$ planes in $\bF_p^3$, assuming the incidence graph is $K_{2, s}$-free.

Finally, it would be also interesting to find other geometric problems for which the methods of this paper can be applied. For example, it may be possible to use our approach to address other non-degeneracy conditions beyond forbidding $K_{s, s}$ in the incidence graph. 

\vspace{0.4cm}

\noindent
\textbf{Acknowledgements.} We would like to thank Larry Guth for suggesting to study point-variety incidences and stimulating discussions. We would also like to thank Joshua Zahl and Misha Rudnev for pointing out additional references.

\end{document}